\documentclass[a4paper,10pt]{article}

\usepackage{amsmath,amssymb,amsthm,graphicx,
	subfig,
verbatim,cite,
} 

\usepackage{geometry}
\usepackage{tikz}
\usetikzlibrary{shapes.geometric,calc,decorations.markings,math}

\tikzstyle{nodo}=[circle,draw,fill,inner sep=0pt,minimum size=%
1.5mm]
\tikzstyle{infinito}=[circle,inner sep=0pt,minimum size=0mm]

\newcommand\R{{\mathbb R}}

\newcommand\N{{\mathbb N}}

\newcommand\T{\mathcal T}

\newcommand\dx{{\,dx}}

\newcommand{\E}{\mathbb{E}}
\newcommand{\V}{\mathbb{V}}

\newcommand\G{\mathcal G}

\newcommand\K{\mathcal K}

\newcommand\HH{\mathcal H}
\newcommand\NN{\mathcal N}

\newcommand\vv{\textsc{v}}

\theoremstyle{definition}

\theoremstyle{plain}
\newtheorem{theorem}{Theorem}[section]
\newtheorem{proposition}[theorem]{Proposition}
\newtheorem{lemma}[theorem]{Lemma}
\newtheorem{corollary}[theorem]{Corollary}

\theoremstyle{remark}
\newtheorem{remark}[theorem]{Remark}
\newtheorem*{remark*}{Remark}

\theoremstyle{definition}
\newtheorem{definition}[theorem]{Definition}

\newcommand\blfootnote[1]{%
	\begingroup
	\renewcommand\thefootnote{}\footnote{#1}%
	\addtocounter{footnote}{-1}%
	\endgroup
}

 \numberwithin{equation}{section}

\date{}

\title{Normalized solutions of $L^2$-supercritical NLS equations \\ on noncompact metric graphs}

\author{Simone Dovetta$^{1,*}$, Louis Jeanjean$^{2,\dagger}$, Enrico Serra$^{1,\ddagger}$ 
	\\ \ \\{\small$^1$Politecnico di Torino, Dipartimento di Scienze
		Matematiche ``G.L. Lagrange'' } \\ {\small
		Corso Duca degli Abruzzi, 24, 10129 Torino, Italy} \\ \ \\
		{\small$^2$Universit\'e  Marie et Louis Pasteur, CNRS, LmB (UMR 6623), F-25000 Besan\c{c}on, France}
		}

\begin{document}

\maketitle

\begin{abstract}
We consider the existence of normalized solutions to nonlinear Schr\"odinger equations on noncompact metric graphs in the $L^2$ supercritical regime. For sufficiently small prescribed mass ($L^2$ norm), we prove existence of  positive solutions on two classes of graphs: periodic graphs and noncompact graphs with finitely many edges and suitable topological assumptions. Our approach is based on mountain pass techniques.  
A key point to overcome the serious lack of compactness is to show that all solutions with small mass have positive energy. To complement our analysis, we prove that this is no longer true, in general, for large masses. To the best of our knowledge, these are the first results with an $L^2$ supercritical nonlinearity extended on the whole graph and unraveling the role of topology in the existence of solutions.
\end{abstract}

\blfootnote{$^*$simone.dovetta@polito.it}  \blfootnote{$^\dagger$louis.jeanjean@univ-fcomte.fr}  \blfootnote{$^\ddagger$enrico.serra@polito.it}

{\small \noindent \text{Statements and Declarations:} the authors have no relevant financial or non-financial interests to disclose.} \\
{\small \noindent \text{Data availability:} data sharing is not applicable to this article as no datasets were generated or analyzed during the current study.}

\section{Introduction}

\noindent Throughout the paper we consider metric graphs $\G=(\V_\G, \E_\G)$ such that
\begin{itemize}
	\item $\G$ is connected and has at most countably many edges;
	\item $\text{deg}(\vv)<\infty$ for every $\vv\in\V_\G$, where $\text{deg}(\vv)$ is the total number of edges incident at the vertex $\vv$;
	\item $e_0:=\displaystyle\inf_{e\in\E_\G}|e|>0$, where $|e|$ is the length of the edge $e$.
\end{itemize}
For simplicity, if a graph $\G$ satisfies these properties, we say that $\G$ is in the class  $\mathbf{G}$, and we write $\G\in\mathbf{G}$. The notion of metric graph is detailed in~\cite{BK}.
A graph in $\mathbf{G}$ is noncompact if at least one of its edges is unbounded (that is, the graph has at least one half-line) or if the total number of edges in the graph is infinite. It is evident that the set of all noncompact graphs in $\mathbf{G}$ is extremely wide and varied, as it contains structures that may exhibit sensibly different behaviours, making thus impossible to handle them all at once. For this reason, in the present paper we will focus on two specific families of noncompact graphs: graphs with {\em finitely many edges} (and thus with at least one half-line, see e.g. Figure \ref{fig:half}) and {\em periodic graphs} (see e.g. Figure \ref{fig:per}), i.e. graphs with infinitely many edges, whose length is uniformly bounded from above, arranged in a periodic pattern (for a rigorous definition of periodic graphs see e.g. \cite[Definition 4.1.1]{BK}). 

\begin{figure}[t]
	\centering
	\includegraphics[width=0.5\textwidth]{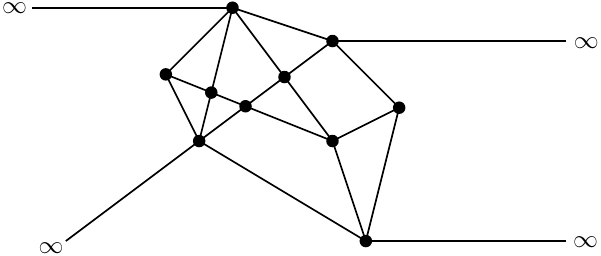}
	\caption{A noncompact graph in $\mathbf{G}$ with finitely many vertices and edges.}
	\label{fig:half}
\end{figure}

\begin{figure}[t]
	\centering
	\begin{tikzpicture}[xscale= 0.4,yscale=0.4]
		\draw (-17,3)--(-5,3); \draw[dashed] (-18,3)--(-17,3); \draw[dashed] (-5,3)--(-4,3); 
		\draw (-17,6)--(-5,6); \draw[dashed] (-18,6)--(-17,6); \draw[dashed] (-5,6)--(-4,6); 
		
		\node at (-17,3) [nodo] {}; \node at (-17,6) [nodo] {};
		\node at (-14,3) [nodo] {}; \node at (-14,6) [nodo] {};
		\node at (-11,3) [nodo] {}; \node at (-11,6) [nodo] {};
		\node at (-8,3) [nodo] {}; \node at (-8,6) [nodo] {};
		\node at (-5,3) [nodo] {}; \node at (-5,6) [nodo] {};
		
		\draw (-17,3)--(-17,6);
		\draw (-14,3)--(-14,6);
		\draw (-11,3)--(-11,6);
		\draw (-8,3)--(-8,6);
		\draw (-5,3)--(-5,6);

		\draw[step=2,thin] (0,0) grid (8,8);
		\foreach \x in {0,2,...,8} \foreach \y in {0,2,...,8} \node at (\x,\y) [nodo] {};
		\foreach \x in {0,2,...,8}
		{\draw[dashed,thin] (\x,8.2)--(\x,9.2) (\x,-0.2)--(\x,-1.2) (-1.2,\x)--(-0.2,\x)  (8.2,\x)--(9.2,\x); }
		\end{tikzpicture}
	\caption{Examples of periodic graphs.}
	\label{fig:per}
\end{figure}
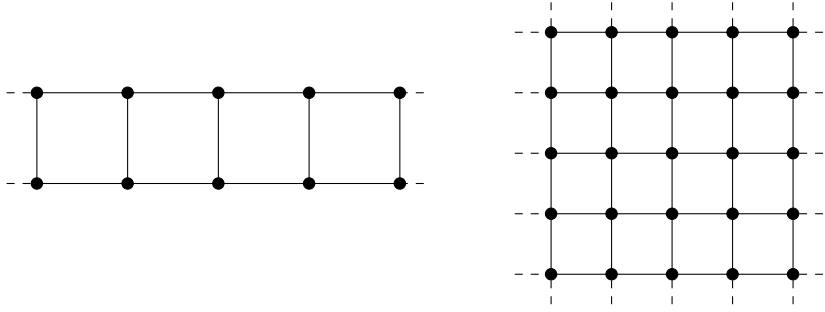

Given a graph $\G\in\mathbf{G}$ and $p>2$, we look for solutions $(u, \lambda)$ of the problem
\begin{equation}
	\label{nlsG}
	\begin{cases}
		u''+u^{p-1}=\lambda u & \text{on every }e\in\E_\G\\
		u\in H^1(\G),\quad u>0 &\text{on }\G\\
		\sum_{e\succ \vv}\frac{du}{dx_e}(\vv)=0 & \text{for every }\vv\in\V_\G\,,
	\end{cases}
\end{equation}
where $e\succ\vv$ indicates that the edge $e$ is incident at the vertex $\vv$, and $\frac{du}{dx_e}(\vv)$ denotes the derivative of $u$ at $\vv$ along the edge $e$ identified with the interval $[0,|e|]$ in such a way that the vertex $\vv$ corresponds to $0$.  

Our focus here is devoted to {\em normalized} solutions of \eqref{nlsG}, i.e. solutions $u$ in the mass constrained space
\[
H_\mu^1(\G):=\left\{v\in H^1(\G)\,:\,\|v\|_{L^2(\G)}^2=\mu\right\},
\]
in the $L^2$-{\em supercritical} regime $p>6$.  To this end, we look for critical points of the energy functional $E:H^1(\G)\to\R$
\[
E(u,\G):=\frac12\|u'\|_{L^2(\G)}^2-\frac1p\|u\|_{L^p(\G)}^p
\]
constrained to $H_\mu^1(\G)$. It is well-known that these critical points
satisfy \eqref{nlsG} (except possibly for the sign condition), for a suitable $\lambda$ that arises as a Lagrange multiplier associated with the mass constraint.

Recently, much effort has been devoted to establish the
existence of normalized solutions of nonlinear Schr\"odinger equations \eqref{nlsG} on metric graphs in the
$L^2$\emph{-subcritical} (i.e., $p\in(2,6)$) or $L^2$\emph{-critical
  regimes} (i.e., $p=6$). In these two cases, the energy functional
$E(\cdot, \G)$ is bounded from below and coercive on the mass
constraint (one needs to require $\mu$ smaller than a threshold when $p=6$).  Various results of existence, non-existence and multiplicity have been obtained (see e.g. \cite{ADST,AST_CMP,AST,ACT,BMP,BDL1,BDL2,DDGS,D,DTjmpa,KMPX,KNP,NP,Pankov,PS,PSV}).

Conversely, in the $L^2$-supercritical regime on general metric graphs,
i.e. when $p>6$, the energy functional $E(\cdot, \G)$ is always
unbounded from below on $H_\mu^1(\G)$ and the approaches developed when $p \leq 6$ are not directly applicable anymore. As a matter of fact, the existence of normalized solutions in the case $p>6$ only started to be considered very recently in \cite{CJS, CGJT, BCJS_Non} when the graph is compact or when the nonlinearity is restricted to a compact subset of the graph.

The main existence results of the paper are stated in the next two theorems.
Here, by {\em pendant} we mean a bounded edge with a vertex of degree 1, whereas by a {\em signpost} we mean the union of a bounded edge and a loop, glued together at a common vertex. 
\begin{theorem}
	\label{thm:exHalf}
	Let $\G\in\mathbf{G}$ be a noncompact metric graph with finitely many edges and at least one pendant or one signpost. Then, for every $p>6$ there exists $\mu_{p,\G}>0$, depending on $p$ and $\G$, such that, for every $\mu\in(0,\mu_{p,\G}]$, problem \eqref{nlsG} admits a solution $u\in H_\mu^1(\G)$ satisfying $E(u,\G) >0$ and $\lambda >0$.
\end{theorem}

\begin{theorem}
	\label{thm:exper}
	Let $\G\in\mathbf{G}$ be a periodic graph. Then, for every $p>6$ there exists $\mu_{p,\G}>0$, depending on $p$ and $\G$, such that, for every $\mu\in(0,\mu_{p,\G}]$, problem \eqref{nlsG} admits a  solution $u\in H_\mu^1(\G)$ satisfying $E(u,\G) >0$ and $\lambda >0$.
\end{theorem}
Theorems \ref{thm:exHalf}--\ref{thm:exper}  are the first results on the existence of normalized solutions to \eqref{nlsG} when $p$ is $L^2$-supercritical, the graph $\G$ is noncompact, and the nonlinearity acts on the whole of $\G$ (whereas without the mass constraint some results can be found e.g. in \cite{CDPS,DDGST}). The only results for supercritical $p$ obtained so far \cite{BCJS_Non,CGJT,CJS} either require $\G$ to be compact or the nonlinearity to be localized on the compact core of the graph. The solutions found in those papers are critical points of mountain pass type for the energy $E(\cdot,\G)$ in the mass constrained space $H_\mu^1(\G)$, and the assumptions on the compactness of the graph or on the localization of the nonlinearity are essentially used to obtain a minimum of compactness needed to prove that Palais--Smale sequences converge. Such an issue, together with the very boundedness of these sequences, is one of the major problems when dealing with supercritical problems. We recall that arguments based on scaling, employed for example for the NLS equation on $\R^N$, are not available on metric graphs since these are not scale invariant.

The solutions found in Theorems \ref{thm:exHalf}--\ref{thm:exper} are again of mountain pass type for $E(\cdot,\G)$ in $H_\mu^1(\G)$, in the same spirit of \cite{BCJS_Non,CJS}. In our results, however, we manage to avoid all compactness assumptions on $\G$ and on the nonlinearity, replacing them with the requirement that the prescribed mass of the solutions we seek be small and exploiting specific topological features of the graphs under exam.

Let us give some elements of our strategy.  As in \cite{BCJS_Non,CJS}, we shall obtain the solutions of Theorems \ref{thm:exHalf}--\ref{thm:exper} as a limit of a sequence of solutions of approximating problems.  To this aim, we first consider the family of functionals $E_\rho:H^1(\G)\to\R$  given by
\[
E_\rho(u,\G):=\frac12\|u'\|_{L^2(\G)}^2-\frac\rho p\|u\|_{L^p(\G)}^p\, \quad \text{where} \quad \rho \in [1/2,1].
\]
The functionals $E_\rho(\cdot,\G)$ have a mountain pass geometry and adapting \cite[Theorem 1.10]{BCJS_TAMS} to our context we show that, for almost every $\rho \in [1/2,1]$, $E_\rho(\cdot,\G)$ admits a {\it bounded} Palais-Smale sequence at the mountain pass level with certain Morse index-type properties.
We prove that this sequence converges and that its limit is a critical point $u_{\rho}$ of $E_\rho(\cdot, \G)$ constrained to $H_\mu^1(\G)$. We then consider a sequence $\rho_n \to 1$ and show that the sequence of corresponding solutions $u_{\rho_n}$ converges to a critical point $u \in H_\mu^1(\G)$ of $E(\cdot, \G)$, which is the solution we are looking for in Theorems \ref{thm:exHalf}--\ref{thm:exper}. 

Whether it is to derive the existence of $u_{\rho}$, for almost every $\rho \in [1/2,1]$, or later in the analysis of the sequence $(u_{\rho_n})_n $, we are led to study the convergence of such Palais-Smale sequences. 
The noncompactness of the domain and the information on the Morse index guarantee that the corresponding Lagrange multipliers are nonnegative. Then, in order to recover the necessary compactness, we rule out the possible partial loss of mass at infinity proving that, on every noncompact graph, all solutions of \eqref{nlsG} have positive energy. This result, which is perhaps of independent interest, requires the small mass assumption and is given in Proposition \ref{asympt}. To conclude, we are  left to exclude that the whole mass runs away at infinity. This is accomplished either exploiting the invariance under translations on periodic graphs, or  suitable topological assumptions on graphs with finitely many edges. Noteworthy, for this final step to be completed, the previous conditions must be combined again with the small mass assumption.

We wish to notice that both the requirement of small mass and on the topology of the graph are far from arbitrary or technical, but reflect instead important features of the NLS equation on graphs. Indeed, it seems to us that existence results for mountain pass type solutions with the nonlinearity acting on the whole structure are likely not to hold on general noncompact graphs when even just one of our assumptions is dropped.

Consider first the topological assumptions we exploited: in Theorem \ref{thm:exHalf} the presence of a {\em pendant} or a {\em signpost}, and in Theorem \ref{thm:exper} the periodicity of $\G$. As already said, these kind of assumptions have been deeply investigated in the subcritical and critical cases (see e.g. \cite{ADST,AST_CMP,AST,ACT,BMP,BDL1,BDL2,DDGS,D,DTjmpa,KMPX,KNP,NP,Pankov}) to obtain existence or nonexistence of ground states and of more general solutions. This is, on the contrary, the first time that topological properties of the graphs are shown to play a role in the supercritical regime. Here they are crucial to provide suitable estimates on the mountain pass level needed to restore compactness.  The relevance of these conditions is highlighted by the fact that for nonperiodic graphs that do not satisfy the assumptions of Theorem  \ref{thm:exHalf} one cannot hope, generally, to prove existence of solutions as in the preceding results. In fact, in Proposition \ref{prop:H} below we show that for graphs satisfying the so--called assumption (H) (a topological condition on graphs with half-lines originally introduced in \cite{AST}) the mountain pass level we construct coincides with the level of the corresponding problem (with the same prescribed mass) on $\R$. However, a prototypical example of graphs fulfilling assumption (H) is the star graph (Figure \ref{fig:star}), for which it is well-known (see e.g. \cite{ACFN}) that, for {\em any} mass $\mu$, no solution of \eqref{nlsG} exists with energy equal to that of the mountain pass level on $\R$. We conjecture that the same is true for {\em all} graphs satisfying assumption (H) but, at the present time, we are not able to prove this statement in its full generality. If a proof of this result were available it would be extremely interesting, since it would reveal a deep property of the NLS equation on graphs: in \cite{AST_CMP,AST}, it is shown that in the subcritical and critical regimes, assumption (H) prevents the existence of ground states; the validity of our conjecture would then provide a similar statement for mountain pass solutions in  the supercritical case.

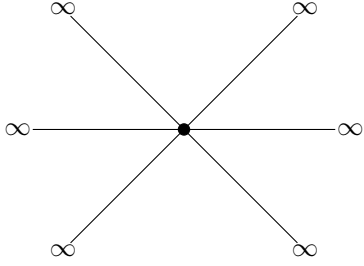
\begin{figure}
	\centering
	\begin{tikzpicture}
		\node at (0,0) [nodo] {};
		\draw (2,0)--(-2,0);
		\node at (2.2,0) [infinito] {$\footnotesize\infty$};
		\node at (-2.2,0) [infinito] {$\footnotesize\infty$};
		\draw (-1.5,-1.5)--(1.5,1.5); 
		\node at (-1.6,-1.6) [infinito] {$\footnotesize\infty$};
		\node at (1.6,1.6) [infinito] {$\footnotesize\infty$};
		\draw (-1.5,1.5)--(1.5,-1.5); 
		\node at (1.6,-1.6) [infinito] {$\footnotesize\infty$};
		\node at (-1.6,1.6) [infinito] {$\footnotesize\infty$};
	\end{tikzpicture}
	\caption{A star graph with $6$ half-lines glued together at their common origin.}
	\label{fig:star}
\end{figure}

The importance of the smallness assumption on the mass in Theorems \ref{thm:exHalf}--\ref{thm:exper} is analogous. As indicated, a key step in our argument is to prove that, for small masses, {\em all} solutions of \eqref{nlsG} have a nonnegative energy. On the contrary, the next two results show that, in general, for large values of $\mu$ there do exist solutions with {\em negative} energy. 
\begin{theorem}
	\label{thm:neg}
	Let $\G\in\mathbf{G}$ be a noncompact graph with finitely many edges, no pendant and exactly one half-line. Then, there exists $\delta>0$ such that, for every $p\in[6,6+\delta]$, there exists $\mu>0$, depending on $p$ and $\G$, for which problem \eqref{nlsG} admits a positive solution $u\in H_\mu^1(\G)$ with $E(u,\G)<0$.
\end{theorem}
\begin{theorem}
	\label{thm:prop_p}
	Let $\G\in\mathbf{G}$ be a noncompact graph with finitely many edges, at least one of which bounded, and such that every vertex of $\G$ is attached to an even number of half-lines (possibly zero). For every $p>6$ there exists $\overline{\mu} > 0$, depending on  $p$ and $\G$, such that, for all $\mu \ge \overline{\mu}$, problem \eqref{nlsG} admits a positive solution $u\in H_\mu^1(\G)$ with $E(u,\G)<0$.
\end{theorem}
As its proof shows, the solutions of Theorem \ref{thm:neg} are obtained as a perturbation of negative energy ground states of $E$ at the critical exponent $p=6$, that on graphs satisfying the hypotheses of Theorem \ref{thm:neg} are known to exist for a full interval of masses (see \cite[Theorem 3.3]{AST_CMP}). On the contrary, the result of Theorem \ref{thm:prop_p} is not perturbative and it was communicated to us by D. Galant \cite{DG} (to whom we are sincerely grateful). Since their energy is negative, such solutions cannot be found by our mountain pass approach, which is designed to work at strictly positive levels. 
Observe also that the simplest graph fulfilling Theorem \ref{thm:neg} is perhaps the {\em tadpole} graph (Figure \ref{fig:tad}), that contains a signpost, whereas the simplest one to which Theorem \ref{thm:prop_p} applies is the so-called  $\mathcal{T}$-graph (Figure \ref{fig:T}), that was extensively studied in \cite{ACT} and that contains a pendant. This shows that, without the small mass assumption, there is no chance to recover our mountain pass argument, even in presence of the topological features introduced above. 

We finally note that the advantage given by the small mass to control from below the energy of solutions was already observed and exploited in some related problems posed on $\R^N$ (see for example \cite{MRV}).

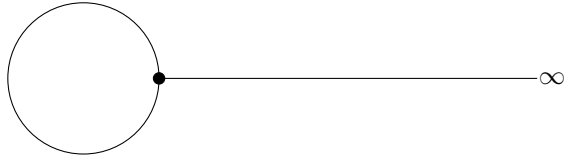
\begin{figure}[t]
	\centering
	\begin{tikzpicture}
		\node at (0,0) [nodo] {};
		\draw (-1,0) circle (1);
		\draw (0,0)--(5,0);
		\node at (5.2,0) [infinito] {$\footnotesize\infty$};
	\end{tikzpicture}
	\caption{The tadpole graph, i.e. a loop attached to a half-line.}
	\label{fig:tad}
\end{figure}
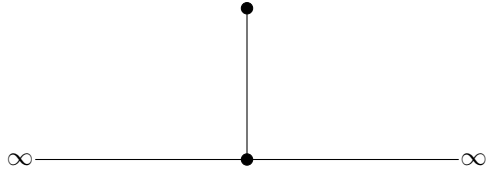
\begin{figure}[t]
	\centering
	\begin{tikzpicture}
				\node at (0,0) [infinito] {$\footnotesize\infty$};
				\draw (0.2,0)--(5.8,0);
				\node at (6,0) [infinito] {$\footnotesize\infty$};
				\node at (3,0) [nodo] {};
				\draw (3,0)--(3,2);
				\node at (3,2) [nodo] {};
			\end{tikzpicture}
	\caption{The $\mathcal{T}$-graph, i.e. a pendant attached to a full line.}
	\label{fig:T}
\end{figure}

\smallskip
The remainder of the paper is organized as follows. Section \ref{sec:prel} collects some preliminary results on the mountain pass approach we will use and on the NLS equation on the real line. Section \ref{sec:lowerE} provides a general lower bound on the energy of positive solutions of \eqref{nlsG} with small mass. In Section \ref{sec:apriori} we introduce the mountain pass problem and we derive some a priori estimates on the mountain pass level depending on the topology of the graphs. Finally, in Sections \ref{sec:exhalf}--\ref{sec:exper}--\ref{sec:thmneg} we  prove Theorem \ref{thm:exHalf}, Theorem \ref{thm:exper} and Theorems \ref{thm:neg}--\ref{thm:prop_p} respectively. 

\smallskip
{\bf Notation.} In what follows, for $u\in L^q(\G)$ we will write $\|u\|_q$ in place of $\|u\|_{L^q(\G)}$, omitting the explicit reference to the domain of integration when it coincides with the whole space. Conversely, the full notation for norms will always be used whenever the domain of integration is a proper subset of the full space.  \bigskip

\section{Preliminaries}
\label{sec:prel}

This section collects various preliminary results that will be largely used in the rest of the paper.

\subsection{A general existence result for bounded Palais-Smale sequences}

The main aim of this subsection is to rewrite in the context of the present paper the content of \cite[Theorem 1.10]{BCJS_TAMS}. To this end, for any $\rho>0$, let $E_\rho:H^1(\G)\to\R$ be given by
\[
E_\rho(u,\G):=\frac12\|u'\|_2^2-\frac\rho p\|u\|_p^p\,.
\]
\begin{remark}
Throughout the paper, the notation $E_\rho$ will be used when $\rho\neq1$ only. When $\rho=1$, we will always write $E$ as in the Introduction, avoiding the symbol $E_1$.
\end{remark}
We recall the following definition.
\begin{definition}
	Let $\G\in\mathbf{G}$, $p>6$, $\mu>0$ and $\rho>0$. For every $u\in H_\mu^1(\G)$ and $\theta \geq 0$, the quantity
	\[
	\widetilde{m}_\theta(u):=\sup\left\{\text{\normalfont dim}\,L: L \text{\normalfont \, subspace of $T_uH_\mu^1(\G)$ such that }D^2E_\rho(u,\G)[w,w]<-\theta\|w\|_{H^1(\G)}^2,\,\forall w\in L\setminus\{0\}\right\}
	\]
	is called the {\em approximate Morse index }of $u$ with respect to $\theta$, where $T_uH_\mu^1(\G)$ is the tangent space to $H_\mu^1(\G)$ at $u$ and 
	\[
	D^2E_\rho(u,\G)[w,w]:=E_\rho''(u,\G)[w,w]-\frac{E_\rho'(u,\G)[u]}{\|u\|_2^2}\| w\|_{L^2(\G)}\qquad\forall w\in T_uH_\mu^1(\G)\,.
	\]
	If $u$ is a critical point of $E_\rho$ constrained to $H_\mu^1(\G)$ and $\theta=0$, we say that $\widetilde{m}_0(u)$ is the {\em Morse index of $u$ as a constrained critical point}.
\end{definition}

The next theorem, that is no more than \cite[Theorem 1.10]{BCJS_TAMS} applied to the family of functionals $E_\rho$ above, provides a general existence result of bounded Palais-Smale sequences for these functionals constrained to $H_\mu^1(\G)$, with further information on their approximate Morse index.
\begin{theorem}
	\label{thm:PS}
	Let $\G\in\mathbf{G}$, $p>6$ and $\mu>0$ be given. Let $I\subset(0,+\infty)$ be a given interval and assume that there exist $A,B\subset H_\mu^1(\G)$ independent of $\rho$ for which
	\[
	c_\rho:=\inf_{\gamma\in\Gamma}\max_{t\in[0,1]}E_\rho(\gamma(t),\G)>\max\left\{\sup_{u\in A}E_\rho(u,\G),\sup_{u\in B}E_\rho(u,\G)\right\}\qquad\forall \rho\in I\,,
	\]
	where 
	\[
	\Gamma:=\left\{\gamma\in C\left([0,1],H_\mu^1(\G)\right)\,:\,\gamma(0)\in A,\gamma(1)\in B\right\}\,.
	\]
	Then, for almost every $\rho\in I$, there exist sequences $(u_n)_n\subset H_\mu^1(\G)$, such that, as $n\to+\infty$, $\zeta_n\to0^+$ and
	\begin{itemize}
		\item[(i)] $\displaystyle E_\rho(u_n,\G)\to c_\rho$;
		\item[(ii)] $\displaystyle E_\rho'(u_n,\G)-\frac{E_\rho'(u_n,\G)[u_n]}{\mu}u_n\to0$ in the dual of $H_\mu^1(\G)$;
		\item[(iii)] $(u_n)_n$ is bounded in $H_\mu^1(\G)$;
		\item[(iv)] $\displaystyle \widetilde{m}_{\zeta_n}(u_n)\leq 1$.
	\end{itemize}
\end{theorem} 
Theorem \ref{thm:PS} will be used as a key tool in the proof of Theorems \ref{thm:exHalf}--\ref{thm:exper}. The information on the Morse index will play a relevant role combined with the following lemma, that simply rephrases in our setting \cite[Lemma 2.5]{BCJS_Non}.
\begin{lemma}
	\label{lem:dim3}
	Let $\G\in\mathbf{G}$, $p>6,\mu>0$ and $\rho>0$. Assume that $(u_n)_n\subset H_\mu^1(\G)$, $(\lambda_n)_n\subset\R$ and $(\zeta_n)_n\subset\R^+$ are such that,  as $n\to+\infty$, $\zeta_n\to0^+$ and
	\begin{itemize}
		\item[(i)] if a subspace $W\subset H^1(\G)$ satisfies
		\[
		E_\rho''(u_n,\G)[w,w]+\lambda_n\|w\|_2^2<-\zeta_n\|w\|_{H^1(\G)}^2\qquad\forall w\in W\setminus\{0\}
		\]
		for sufficiently large $n$, then $\text{\normalfont dim}\,W\leq2$, and
		
		\item[(ii)] there exist $\lambda\in\R$, a subspace $Y\subset H^1(\G)$ with $\text{\normalfont dim}\,Y\geq3$, and $a>0$ such that, for every large $n$,
		\[
		E_\rho''(u_n,\G)[w,w]+\lambda\|w\|_2^2\leq-a\|w\|_{H^1(\G)}^2\qquad\forall w\in Y.
		\]
	\end{itemize}
	Then $\lambda_n>\lambda$ for every $n$ large enough.
\end{lemma}

\subsection{NLS equations on the real line}
For every $\mu,\rho>0$, it is well-known that there exists a unique solution $(\phi_{\mu,\rho},\lambda_{\mu,\rho})\in H_\mu^1(\R)\times\R^+$ of the problem
\begin{equation}
\begin{cases}
\label{eq:nlsR}
u''+\rho u^{p-1}=\lambda u & \text{on }\R\\
u>0,\quad u(0)=\max_{x\in\R}u(x)\,.
\end{cases}
\end{equation}
Setting 
\begin{equation}
\label{eq:alfabeta}
\alpha=\frac2{6-p}\,,\qquad\beta=\frac{p-2}{6-p}\,,
\end{equation}
this solution can be written explicitly as
\[
\phi_{\mu,\rho}(x)=\rho^{\frac{\alpha}{2}}\phi_{\mu,1}\left(\rho^\alpha x\right)\qquad\forall x\in\R,
\]
where
\[
\phi_{\mu,1}(x)=\mu^\alpha \phi_{1,1}(\mu^\beta x)\qquad\forall x\in\R
\]
and
\[
\phi_{1,1}(x)=C_p\text{\normalfont sech}^{\frac\alpha\beta}(c_p x)\,,
\]
with $C_p,c_p>0$ suitable constants depending only on $p$. In particular, the previous relations show that, for every  $p>6$ and $x\in\R$, as $\mu\to 0$ it holds
\begin{equation}
\label{eq:asphi}
\begin{split}
\phi_{\mu,\rho}(x)&\,= C_p'\rho^{\frac\alpha2}\mu^\alpha e^{-\frac{2c_p}{p-2}\rho^\alpha\mu^\beta x}+o\left(\mu^\alpha e^{-\frac{2c_p}{p-2}\rho^\alpha\mu^\beta x}\right)\\
\phi_{\mu,\rho}'(x)&\,=C_p''\rho^{\frac{3\alpha}2}\mu^{\alpha+\beta} e^{-\frac{2c_p}{p-2}\rho^\alpha\mu^\beta x}+o\left(\mu^{\alpha+\beta} e^{-\frac{2c_p}{p-2}\rho^\alpha\mu^\beta x}\right)
\end{split}
\end{equation}
for suitable constants $C_p',C_p''>0$ depending only on $p$. Furthermore, direct computations yield the identity
\begin{equation}
\label{eq:relphi1}
\|\phi_{\mu,\rho}'\|_{2}^2=\frac{p-2}{2p}\rho\|\phi_{\mu,\rho}\|_p^p\,,
\end{equation}
so that
\begin{equation}
\label{eq:Erhophi}
E_\rho(\phi_{\mu,\rho},\R)=\frac{p-6}{4p}\rho\|\phi_{\mu,\rho}\|_p^p\,.
\end{equation}
Recalling that $E(\phi_{\mu,1},\R)=\mu^{2\beta+1}E(\phi_{1,1},\R)>0$, it then follows that
\begin{equation}
\label{eq:Emu}
E_\rho(\phi_{\mu,\rho})=\theta_p\rho^{\frac4{6-p}}	\mu^{2\beta+1} \quad \mbox{with} \quad \theta_p:=\frac{p-6}{4p}\|\phi_{1,1}\|_p^p>0.
\end{equation}
In particular \eqref{eq:Emu} implies that $E_\rho(\phi_{\mu,\rho})$ is a strictly decreasing function of $\mu >0$.

\subsection{Gagliardo-Nirenberg inequalities on graphs}
For every noncompact graph $\G\in\mathbf{G}$, we recall here the following well-known Gagliardo-Nirenberg inequalities
\begin{equation}
\label{eq:GNp}
\|u\|_{p}^p\leq K_{p,\G}\|u\|_{2}^{\frac p2+1}\|u'\|_{2}^{\frac p2-1}\qquad\forall u\in H^1(\G),\,p>2\,,
\end{equation}
\[
\|u\|_{\infty}\leq \sqrt{2}\|u\|_{2}^{\frac12}\|u'\|_{2}^{\frac12}\qquad\forall u\in H^1(\G)\,,
\]
with $K_{p,\G}>0$ depending on $p$ and $\G$ only (for a proof of these inequalities see e.g. \cite[Section 2]{AST_JFA}).

\subsection{Properties of the Lagrange multiplier}
For $\G\in\mathbf{G}$, let 
\begin{equation}
\label{eq:lambdaG}
\lambda_\G:=\inf_{u\in H^1(\G)\setminus\{0\}}\frac{\|u'\|_{2}^2}{\|u\|_{2}^2}
\end{equation}
be the bottom of the spectrum of the operator $\displaystyle-\frac{d^2}{dx^2}$ on $\G$ endowed with Kirchhoff conditions at the vertices.
Introducing the shorthand notation
\[
S_\mu =\left\{ u\in H_\mu^1(\G) \mid u \text{ solves } \eqref{nlsG} \text{ for some } \lambda \in \R \right\},
\]
we have the following lemma.
\begin{lemma}
	\label{Properties_L}
	Let $\G\in\mathbf{G}$. There results
	\begin{itemize}
		\item[(a)] if $u \in S_\mu$, then $\lambda \geq  - \lambda_\G$;
		\item[(b)] if $\G$ is a graph with at least one half-line or a periodic graph,  then $\lambda_\G=0$; 
		\item[(c)] if $\G$ is a noncompact graph with finitely many edges and $u  \in S_{\mu}$ for some $\mu >0$, then $\lambda >0$; 
		\item[(d)] if $\G$ is a graph with at least one half-line or a periodic graph, then for any $m >0$ there exists $\mu_m  >0$ such that, if $u \in S_\mu$ with $\mu \in (0, \mu_m]$, then $\lambda \geq m$.
	\end{itemize}
\end{lemma}

To prove Lemma \ref{Properties_L}(d), we shall need the next preliminary lemma.
\begin{lemma}
	\label{growth}
	Let $u \in S_\mu$. Then 
	\begin{equation}
	\label{apriori}
	\|u\|_{\infty} \le \left(\frac{p}2 \lambda + \frac{p\pi^2}{2e_0^2}\right)^{\frac1{p-2}}.
	\end{equation}
\end{lemma}

\begin{proof}
	This can be checked by standard phase plane considerations as follows. Let $M := \|u\|_{\infty} >0$. If $M^{p-2} \le p\lambda/2$, there is nothing to prove, so we only deal with the case $M^{p-2} > p\lambda/2$. Let then $M$ be achieved in some edge $e$ at a point called $q$. Reversing, if necessary, the orientation of the coordinate along $e$, we can assume $u$ to be nonincreasing on an interval $[q,r]$ of length $a \ge e_0/2$. Indeed, if no such interval existed, then $u$ would attain a local minimum in the interior of $e$, which is impossible since $u>0$ on $\G$, while the condition  $M^{p-2} > p\lambda/2$ forces all local minima of $u$ to be strictly negative.
	
	Since $u$ solves \eqref{nlsG}, for every $x\in[q,r]$ there results
	\[
	\frac12 |u'(x)|^2 + \frac1p |u(x)|^p - \frac{\lambda}2 |u(x)|^2 =  \frac1p |u(q)|^p - \frac{\lambda}2 |u(q)|^2 =  \frac1p M^p - \frac{\lambda}2 M^2>0. 
	\]
	Then
	\[
	a = \int_q^r dx = \int_{u(r)}^M \frac{du}{\sqrt{\frac2p\left(M^p -u^p\right)- \lambda\left(M^2-u^2\right)}}
	\]
	or, setting $t = u/M$,
	\[
	a =  \int_{u(r)/M}^1 \frac{dt}{\sqrt{\frac2p M^{p-2}\left(1 -t^p\right)- \lambda\left(1-t^2\right)}}.
	\]
	Now, as $t \in (0,1]$, we have $1-t^p \ge 1-t^2$, so that 
	\[
	a \le  \frac1{\sqrt{\frac2p M^{p-2} -\lambda}}\int_{u(r)/M}^1 \frac{dt}{\sqrt{1-t^2}}\le \frac1{\sqrt{\frac2p M^{p-2} -\lambda}}\int_0^1 \frac{dt}{\sqrt{1-t^2}}
	= \frac{\pi}{2\sqrt{\frac2p M^{p-2}-\lambda}}\,,
	\]
	namely
	\[
	M^{p-2} \le \frac{p}2 \lambda + \frac{p\pi^2}{8a^2} \le \frac{p}2 \lambda + \frac{p\pi^2}{2e_0^2},
	\]
	since $a \ge e_0/2$.
\end{proof}

\begin{proof}[Proof of Lemma \ref{Properties_L}]
Point (a) is somehow standard, but we prove it for completeness. Let $\K$ be a given connected, bounded subset of $\G$, and let $\varphi$ solve
\begin{equation*}
	%\label{nlsG}
	\begin{cases}
		-\varphi''=\lambda_1(\K) \varphi  & \text{on every edge of } \K\\
		\varphi >0 &\text{on }\K\\
		\sum_{e\succ \vv}\frac{d \varphi}{dx_e}(\vv)=0 & \text{for every vertex }\vv \text{ in } \K \backslash \partial \K\\
		\varphi(x) = 0 & \text{for }x \in \partial \K,
	\end{cases}
\end{equation*}
with $\lambda_1(\K)$  the first eigenvalue of $\displaystyle-\frac{d^2}{dx^2}$ on $\K$ with the above boundary conditions. Let $u \in H^1(\G)$ solve
\begin{equation*}
	%\label{nlsG}
	\begin{cases}
		u'' + u^{p-1 }=\lambda u  & \text{on every } e\in\E_{\G}\\
		u >0 &\text{on }\G\\
		\sum_{e\succ \vv}\frac{d u}{dx_e}(\vv)=0 & \text{for every }\vv \in\V_\G,
	\end{cases}
\end{equation*}
for some $\lambda \in \R$. Multiplying the equation of $u$ by $\varphi$ and the equation of $\varphi$ by $u$  we get
$$- \int_{\K}u' \varphi ' \,dx+ \int_{\K} u^{p-1} \varphi\,dx = \lambda \int_{\K} u \varphi \quad\,dx \quad\text{and} \quad
-  \varphi' u_{| \partial \K} + \int_{\K}\varphi' u'\,dx = \lambda_1(\K) \int_{\K} u \varphi\,dx.$$
Note that $- \varphi' u_{| \partial \K} >0$, so that $ - \int_{\K}\varphi'u'\,dx > - \lambda_1(\K) \int_{\K}u \varphi\,dx.$ Coupling with the other equation yields
\[
\int_{\K} u^{p-1} \varphi\,dx  < (\lambda + \lambda_1(\K)) \int_{\K} u \varphi\,dx.
\]
Since the integrals are positive, $\lambda > - \lambda_1(\K)$. By the arbitrariness of $\K$, this gives $\lambda \geq - \lambda_{\G}.$
The proof of (b)  is evident for graphs with half-lines, whereas for periodic graphs we refer e.g. to \cite[Appendix A]{BDS}. To prove (c) it is enough to observe that, if $u\in S_\mu$, then it is strictly positive everywhere on $\G$ by definition. This is impossible if $\lambda=0$, because in this case all nonzero solutions on a half-line are restrictions of periodic functions (and thus not in $L^2(\G)$). 
To prove (d), let  $m>0$ be arbitrary and assume by contradiction that there exist sequences $\mu_n \to 0$ and $(u_n)_n \subset S_{\mu_n}$ solving \eqref{nlsG} with $\lambda_n$ such that $0 \leq \lambda_n <m$.
	  Since $u_n \in S_{\mu_n}$,
	\begin{equation}
	\label{interm}
	\|u'_n\|_{2}^2 \le  \|u'_n\|_{2}^2 +\lambda_n \|u_n\|_{2}^2 = \|u_n\|_{p}^p\,,
	\end{equation}
	which coupled with \eqref{eq:GNp} yields
	\[
	\|u'_n\|_{2}^2 \le  \|u_n\|_{p}^p \le  K_{p,\G}\mu_n^{\frac{p+2}{4}}\|u_n'\|_{2}^{\frac{p-2}{2}},
	\]
	namely
	\[
	1 \le K_{p,\G}\mu_n^{\frac{p+2}{4}}\|u_n'\|_{2}^{\frac{p-6}{2}}.
	\]
	Since $\mu_n \to 0$, this shows that $\|u_n'\|_{2} \to +\infty$, and then the same for $\|u_n\|_{p}$, by \eqref{interm}. Hence, by $\|u_n\|_{p}^p \le \|u_n\|_{\infty}^{p-2}\mu_n$, we see that $ \|u_n\|_{\infty} \to \infty$, and the contradiction is achieved  invoking \eqref{apriori}.
\end{proof}

\section{A general energy estimate for positive solutions}
\label{sec:lowerE}

The aim of this section is to prove an asymptotic estimate on the energy of small mass positive solutions of \eqref{nlsG}. 
The main result of the section is the following.
\begin{proposition}
	\label{asympt} Let $\G\in\mathbf{G}$ be either a noncompact graph with finitely many edges or a periodic graph. If $\G$ has no pendant, then
	\begin{equation}
	\label{ge1}
	\liminf_{\mu \to 0} \inf_{u \in S_\mu} \frac{E(u,\G)}{E(\phi_{\mu,1},\R)} \ge 1.
	\end{equation}
	If $\G\in{\mathbf{G}}$ has at least one pendant, then
	\begin{equation}
	\label{ge1pend}
	\liminf_{\mu \to 0} \inf_{u \in S_\mu} \frac{E(u,\G)}{E(\phi_{2\mu,1},\R^+)} \ge 1.
	\end{equation}
\end{proposition}

\begin{proof}
	We start proving \eqref{ge1}. If it is false, there exist $\delta \in (0,1)$, a sequence $\mu_n \to 0$ and a sequence $(u_n) \subset S_{\mu_n}$ such that, for every $n$, 
	\begin{equation}
	\label{contr}
	\frac{E(u_n,\G)}{E(\phi_{\mu_n,1},\R)} \le 1-\delta.
	\end{equation}
	Each $u_n$ solves \eqref{nlsG} for some $\lambda_n \geq 0$. In view of Lemma \ref{Properties_L}(d) we have that $\lambda_n \to +\infty$ as $n \to \infty$.
	To proceed, consider a point $q_n$ in some edge $e_n$ where $u_n$ attains a local maximum. Since $u_n$ solves \eqref{nlsG}, it must be $u_n(q_n) \ge \lambda_n^{\frac{1}{p-2}}$. If equality holds,  $u_n$ is constant on $e_n$ by uniqueness in the Cauchy problem, and $e_n$ is a bounded edge. But then
	\[
	\mu_n = \int_\G |u_n|^2\dx \ge  \int_{e_n} |u_n|^2\dx = \lambda_n^{\frac{2}{p-2}}|e_n| \ge  \lambda_n^{\frac{2}{p-2}}|e_0|,
	\]
	which is impossible since $\lambda_n \to +\infty$ and $\mu_n \to 0$. Therefore, $u_n$ is strictly larger than $\lambda_n^{\frac{1}{p-2}}$ at all its local maximum points, which are therefore nondegenerate. Since $u_n \in H_{\mu_n}^1(\G)$, outside a compact set $K_n \subset \G$, $u_n$ is smaller than some number, say $1$ (certainly smaller than $\lambda_n^{\frac{1}{p-2}})$, and thus all local maximum points of $u_n$ are in $K_n$. Since by the preceding argument they are also isolated, it then follows that each $u_n$ can have at most a finite number of maximum points.
	
	Since $\mu_n \to 0$, we can also assume that $\min_e u_n < 1$ for every $e \in \E_\G$ and every $n$.
	We then set
	\[
	A_n = \{x \in \G \mid u_n(x) >1\}
	\]
	and we note that no local minimum point of $u_n$ can be in $A_n$.
	Each $A_n$  is the disjoint union of a finite number of connected open subsets $B_1^n,\dots, B_{k_n}^n$ of $\G$ such that 
	\begin{itemize}
		\item[(i)] $|B_i^n| \to 0$ for every $i=1,\dots,k_n$ as $n\to \infty$;
		\item[(ii)] each $B_i^n$ contains at least one local maximum point of $u_n$;
		\item[(iii)] each $B_i^n$ contains at most one vertex of $\G$;
		\item[(iv)] if $B_i^n$ contains more than one local maximum point of $u_n$, then it contains exactly one vertex of $\G$.
	\end{itemize}
	Indeed, (i) is obvious as $u_n \ge 1$ on each $B_i^n$ and $\mu_n \to 0$. Moreover, 
	since each $B_i^n$ is non-empty, bounded, contained in $A_n$ and $u = 1$ on $\partial B_i^n$, it contains at least one local maximum point of $u_n$, i.e. property (ii) holds.
	Property (iii) follows by the fact that, if some $B_i^n$ contains at least two vertices of $\G$, then, being connected, it contains at least a whole edge of $\G$, which is impossible because the minimum of $u_n$ on each edge of $\G$ is strictly smaller than 1.
	Finally, to prove (iv),  observe that, if two local maximum points of $u_n$ belong to the same edge of $\G$, then the fact that $u_n$ solves \eqref{nlsG} implies that $u_n$ has a local minimum point halfway between the two local maximum points. Moreover, by phase plane analysis it is easy to see that the value attained by $u_n$ at this local minimum point coincides with the minimum value of $u_n$ on the whole edge. Hence, $u_n$ is smaller than or equal to $1$ at least at one point between two consecutive local maximum points belonging to the same edge, in turn implying that these local maximum points do not belong to the same $B_i^n$. Therefore if $B_i^n$ contains at least two local maximum points, these must belong to different edges. By connectedness, then, $B_i^n$ must contain a vertex, and by (iii) this vertex is unique.
	
	Now, since $0<u_n\le 1$ on $\G\setminus A_n$, it follows that 
	\begin{equation}
	\label{eq:Ancomp}
	E(u_n,\G\setminus A_n) \ge  -\frac1p  \|u_n\|_{L^p(\G\setminus A_n)}^p \ge  - \frac1p  \|u_n\|_{L^2(\G\setminus A_n)}^2 \ge -\frac1p \mu_n = o(1)
	\end{equation}
	as $n \to \infty$.
	
	For every $n$, we denote shortly by $B_n$ a set $B_{i_n}^n$ such that
	\begin{equation}
	\label{Bn}
	E(u_n, B_{i_n}^n) = \min_i E(u_n, B_i^n),
	\end{equation}
	which exists since for every $n$ the number of $B_i^n$'s is finite. Up to subsequences, we can assume that either every $B_n$ contains exactly one vertex $\vv_n$, or that no $B_n$ contains a vertex, and we treat the two cases separately.
	
	Let us first suppose that $B_n$ contains no vertex of $\G$, for every $n$. In this case, for each $n$, $B_i^n$ is an interval $(-a_n,a_n)$ contained in an edge $e_n$ and the restriction of $u_n$ to $[-a_n,a_n]$ is strictly positive,  symmetric with respect to $0$ and decreasing, say,  on $[0,a_n]$.  It satisfies
	\[
	\begin{cases}
	u_n''+u_n^{p-1}=\lambda_n u_n & \text{on }[-a_n,a_n]\,,\\
	u_n(0)= \|u_n\|_{L^\infty(-a_n,a_n)} > \lambda_n^{\frac{1}{p-2}},& \\ 
	u_n(\pm a_n) =1 &
	\end{cases}
	\]
	and 
	\[
	\frac12 |u_n'(x)|^2 + \frac1p |u_n(x)|^p - \frac{\lambda_n}2 |u_n(x)|^2 = C_n \qquad\text{ for every } x \in [-a_n,a_n]
	\]
	for some constant $C_n$. By Lemma \ref{growth},  we see that, as $n \to \infty$,
	\begin{equation}
	\label{Cn}
	C_n  = \frac12  |u_n(0)|^2\left(\frac2p  |u_n(0)|^{p-2} -\lambda _n \right)
	\le \frac12  |u_n(0)|^2 \left(\lambda_n + \frac{\pi^2}{e_o^2} -\lambda _n \right) = 
	O\left(\lambda_n^{\frac{2}{p-2}}\right).
	\end{equation}
	We now set 
	\[
	v_n(x):=\lambda_n^{-\frac1{p-2}}u_n\left(x/\sqrt{\lambda_n}\right),
	\]
	so that $v_n$ is a strictly positive function defined on the interval $\left[-\sqrt{\lambda_n}a_n,\sqrt{\lambda_n}a_n\right]$, symmetric with respect to $0$ and  decreasing on $\left[0,\sqrt{\lambda_n}a_n\right]$. It satisfies
	\begin{equation}
	\label{eq:v1}
	\begin{cases}
	v_n''+v_n^{p-1}=v_n & \text{on }\left[-\sqrt{\lambda_n}a_n,\sqrt{\lambda_n}a_n\right]\,,\\
	v_n(0)= \|v_n\|_{L^\infty\left(-\sqrt{\lambda_n} a_n,\sqrt{\lambda_n} a_n\right)}\ge 1, &\\ 
	v_n\left(\pm \sqrt{\lambda_n} a_n\right) =\lambda_n^{-\frac1{p-2}} & 
	\end{cases}
	\end{equation}
	and
	\[
	\frac12 |v_n'(x)|^2 + \frac1p |v_n(x)|^p - \frac12 |v_n(x)|^2 = \lambda_n^{-\frac{p}{p-2}} C_n \qquad\text{ for every } x \in \left[-\sqrt{\lambda_n}a_n,\sqrt{\lambda_n}a_n\right].
	\]
	Then, recalling \eqref{Cn}, we deduce that, as $n\to \infty$,
	\[
	\left|v'\left(\pm \sqrt{\lambda_n}a_n\right)\right|^2 =  2\lambda_n^{-\frac{p}{p-2}} C_n + \left|v\left(\pm \sqrt{\lambda_n}a_n\right)\right|^2 - \frac2p \left|v\left(\pm \sqrt{\lambda_n}a_n\right)\right|^p = O\left(\lambda_n^{-1}\right) + o(1) = o(1).
	\]
	Summing up, $v_n$ satisfies \eqref{eq:v1} and  $\left(v_n(\pm(\sqrt{\lambda_n}a_n)),v_n'(\pm\sqrt{\lambda_n}a_n)\right) \to (0,0)$  as $n\to+\infty$. By continuity in the phase plane, it then follows that  $\sqrt{\lambda_n}a_n \to \infty$ and that
	\begin{equation}
	\label{scale1}
	E\left(v_n,[-\sqrt{\lambda_n}a_n,\sqrt{\lambda_n}a_n]\right)=E(\phi,\R)+o(1)\qquad\text{as }n\to+\infty\,,
	\end{equation}
	where $\phi$ is the unique solution in $H^1(\R)$ of 
	\begin{equation}
	\label{eq:nlsR1}
	\begin{cases}
	\phi''+\phi^{p-1}=\phi & \text{on }\R\\
	\phi(0)=\|\phi\|_{L^\infty(\R)} \ge 1.
	\end{cases}
	\end{equation}
	Note that $\phi=\phi_{\nu,1}$ for $\displaystyle\nu:=\lim_{n\to+\infty}\|v_n\|_{L^2\left(-\sqrt{\lambda_n} a_n,\sqrt{\lambda_n} a_n\right)}^2>0$. Since by direct computations 
	\[
	E(u_n,B_n)=\lambda_n^{\frac{p+2}{2(p-2)}}E\left(v_n,[-\sqrt{\lambda_n}a_n,\sqrt{\lambda_n}a_n]\right),
	\]
	by \eqref{scale1} it follows
	\[
	\begin{split}
	E(u_n,B_n) \ge \lambda_n^{\frac{p+2}{2(p-2)}}E(\phi_{\nu,1}, \R) + o\left( \lambda_n^{\frac{p+2}{2(p-2)}}\right) = &\,E(\phi_{\nu_n,1},\R) + o( E(\phi_{\nu_n,1},\R)) \\
	\ge&\, E(\phi_{\mu_n,1},\R) + o( E(\phi_{\mu_n,1},\R)),
	\end{split}
	\]
	the last inequality coming by the fact that $\nu_n:=\lambda_n^{\frac{6-p}{2(p-2)}}\nu\leq\mu_n$ by definition of $\nu$ and 
	\[
	\|v_n\|_{L^2\left(-\sqrt{\lambda_n} a_n,\sqrt{\lambda_n} a_n\right)}^2=\lambda_n^{\frac{p-6}{2(p-2)}}\|u_n\|_{L^2(B_n)}^2\leq\lambda_n^{\frac{p-6}{2(p-2)}}\mu_n\qquad\forall n\,.
	\] 
	Hence, by \eqref{eq:Ancomp},
	\[
	E(u_n, \G) =  E(u_n, A_n) + E(u_n, \G\setminus A_n)  \ge    E(u_n, B_n) +o(1) \ge  E(\phi_{\mu_n,1},\R) + o( E(\phi_{\mu_n,1},\R)),
	\]
	so that
	\[
	\frac{E(u_n, \G) }{ E(\phi_{\mu_n,1},\R) } \ge 1 + o(1),
	\]
	which violates \eqref{contr} when $n\to \infty$ and completes the proof in case $B_n$ contains no vertex of $\G$.
	
	We now assume (keeping in mind \eqref{Bn}), that every $B_n$ contains a (unique) vertex $\vv_n$. Thus, each $B_n$ is the union of 
	$k_n \ge 3$ line segments glued together at this vertex.
	
	Hence, letting $k_n:=\text{deg}(\vv_n)$, there exist $a_1^n,\dots,a_{k_n}^n>0$ such that $B_n$ can be identified with the intervals $[0,a_1^n),\dots,[0,a_{k_n}^n)$ glued together at the origin (identified with $\vv_n$). Since by construction $B_n$ contains at least one local maximum point of $u_n$ on $\G$, there are $i_n \ge 0$ intervals $[0,a_1^n),\dots,[0,a_{i_n}^n)$ on which $u_n$ is increasing from $0$ to a local maximum point and then decreasing from this maximum point to the end of the interval, whereas on the intervals $[0,a_{i_n+1}^n),\dots,[0,a_{k_n}^n)$ $u_n$ is decreasing from $0$ to the end of the interval. Note that $i_n<k_n$, because otherwise $u_n$ would have a local minimum point at $\vv_n$, and by the Kirchhoff condition this would imply $u_n(\vv_n)\le 1$, contradicting $\vv_n\in B_n$.
	
	Observe that, if
	\begin{equation}
	\label{eq:uvpiccola}
	\limsup_{n\to+\infty} \lambda_n^{-\frac1{p-2}}u_n(\vv_n)=0,
	\end{equation}
	arguing as in the previous part of the proof on each interval (if any) $[0,a_i^n]$, $i = 1, \dots, i_n$,  yields, as $n\to+\infty$,
	\[
	E\left(u_n,[0,a_i^n]\right) = \lambda_n^{\frac{p+2}{2(p-2)}}E(\phi_{\nu,1},\R)+o\left(\lambda_n^{\frac{p+2}{2(p-2)}}\right)
	\] 
	for $\nu\leq\lambda_n^{\frac{p-6}{2(p-2)}}\mu_n$. On the other hand, on each interval $[0,a_i^n]$ with $i = i_n+1, \dots, k_n$, if $a_i^n = o\left(1/\sqrt{\lambda_n}\right)$, then
	\[
	E(u_n,[0,a_i^n]) \ge -\int_0^{a_n}|u_n|^p\dx \ge o \left(\lambda_n^{\frac{p}{p-2}}\right)a_i^n = o \left(\lambda_n^{\frac{p+2}{2(p-2)}}\right).
	\]
	If, on the contrary, $\sqrt{\lambda_n}a_i^n$ is bounded away from 0, then $E(u_n, [0,a_i^n])>0$ (this can be seen e.g. combining the scaling procedure used in the previous part of the proof with \cite[Lemma 4.5]{DTjmpa}). All in all, this yields, as above, 
	\[
	E(u_n, B_n)   \ge  E(\phi_{\mu_n,1},\R) + o( E(\phi_{\mu_n,1},\R)).
	\]
	Conversely, if \eqref{eq:uvpiccola} does not hold, there exists $K>0$ 
	such that along a subsequence (not relabeled) we have
	\[
	\lambda_n^{-\frac1{p-2}}u_n(\vv_n) \geq K\qquad\forall n\,.
	\]
	Considering the restriction of $u_n$ to $\overline{B_n}$, set then $v_n(x):=\lambda_n^{-\frac1{p-2}}u_n(x/\sqrt{\lambda_n})$. The function $v_n$ is defined on a star graph with $k_n$ bounded edges, identified with the intervals $\left[0,\sqrt{\lambda_n}a_i^n\right]$, $i= 1, \dots, k_n$. Moreover, on each edge, $v_n$ is either increasing from $0$ to a local maximum point and then decreasing to the end of the edge or decreasing, and it satisfies
	\[
	\begin{cases}
	v_n''+v_n^{p-1}=v_n & \text{on each } \left[0,\sqrt{\lambda_n}a_i^n\right]\\
	v_n(0) \ge K  & \\
	\sum_{i=1}^{k_n}\frac{dv_n}{dx_i}(0)=0,
	\end{cases}
	\]
	where $\displaystyle\frac{dv_n}{dx_i}(0)$ denotes the outgoing derivative of $v_n$ at $0$ along $\left[0,\sqrt{\lambda_n}a_i^n\right]$. Furthermore, $v_n\left(\sqrt{\lambda_n}a_i^n\right)\to0$ as $n\to+\infty$ for every $i$.

	Arguing exactly as in the previous part of the proof, we obtain again that  $|v_n'(\sqrt{\lambda_n}a_j^n)|\to0$ and $\sqrt{\lambda_n}a_i^n\to+\infty$ as $n\to+\infty$ for every $i$. By continuity in the phase plane, this implies that the restriction of $v_n$ to each interval $[0,a_i^n)$ converges to the restriction of the function $\phi$ of \eqref{eq:nlsR1} to a suitable interval of the form $[b_i,+\infty)$. Together with the fact that $v_n$ satisfies the Kirchhoff condition at the vertex of the star graph, this implies that 
	\[
	E\left(v_n,\sqrt{\lambda_n}B_n\right)=E(w_n,S_{k_n})+o(1)\qquad\text{as }n\to+\infty\,,
	\]
	where $S_{k_n}$ is the infinite star graph with $k_n$ half-lines and $w_n$ is a $H^1(S_{k_n})$ positive solution of 
	\[
	\begin{cases}
	w_n''+w_n^{p-1}=w_n & \text{on }S_{k_n}\\
	w_n(0) \ge K & \\
	\sum_{i=1}^{k_n}\frac{dw_n}{dx_i}(0)=0.
	\end{cases}
	\]
	Since $k_n\geq 3$ for every $n$, it is well known that $E(w_n,S_{k_n}) \ge E(\phi_{\nu_n,1},\R)$ (see e.g. \cite{ACFN}), where $\nu_n \le \mu_n$ is the mass of $w_n$.
	Hence, recalling that $\displaystyle E(u_n,B_n)=\lambda_n^{\frac{p+2}{2(p-2)}}E\left(v_n,\sqrt{\lambda_n}B_n\right)$, we conclude as in the previous case. This completes the proof of \eqref{ge1}.
	
	The proof of \eqref{ge1pend} is completely analogous to that of \eqref{ge1}. The only difference is that if $u_n$ ``concentrates'' on a pendant, then one must use the half-soliton $\phi_{2\mu_n,1}$ and its level $E(\phi_{2\mu_n,1}, \R^+) <E(\phi_{\mu_n,1},\R)$ in the asymptotic estimates.
\end{proof}
\begin{remark}
	\label{rem:estErho}
	If $\G$ has no pendant, it is clear that the proof of Proposition \ref{asympt}  can be easily modified to check that
	\[
	\liminf_{\mu\to0}\inf_{u\in S_{\mu,\rho}}\frac{E_\rho(u,\G)}{E_\rho(\phi_{\mu,\rho},\R)}\geq1\qquad\forall \rho\in\left[1/2,1\right],
	\]
	where $S_{\mu,\rho}$ is the set of all solutions of the problem
	\[
	\begin{cases}
	u''+\rho u^{p-1}=\lambda u & \text{on every }e\in\E_\G\\
	u\in H_\mu^1(\G),\quad u>0 &\text{on }\G\\
	\sum_{e\succ \vv}\frac{du}{dx_e}(\vv)=0 & \text{for every }\vv\in\V_\G\,.
	\end{cases}
	\]
	Observe that, since $u\in S_\mu$ if and only if $\rho^{-\frac1{p-2}}u\in S_{\rho^{-\frac2{p-2}}\mu,\rho}$, this implies that, given $\G$ with no pendant, $p>6$ and $\varepsilon>0$, there exists $\overline{\mu}_{\G,p,\varepsilon}>0$ independent of $\rho\in\left[1/2,1\right]$ for which 
	\[
	\inf_{u\in S_{\mu,\rho}}\frac{E_\rho(u,\G)}{E_\rho(\phi_{\mu,\rho},\R)}\geq 1-\varepsilon\qquad\forall\mu\in\left(0,\overline{\mu}_{\G,p,\varepsilon}\right],\,\forall \rho\in\left[1/2,1\right]\,.
	\]
	Analogously to Proposition \ref{asympt}, the same is true for graphs with a pendant, replacing $E_\rho(\phi_{\mu,\rho},\R)$ with $E_\rho(\phi_{2\mu,\rho},\R^+)$.
\end{remark}

Recalling (see \eqref{eq:Emu}) that both $E(\phi_{\mu,1},\R) \to + \infty$ and $E(\phi_{2\mu,1},\R^+) \to + \infty$ as $\mu \to 0$, the following is a direct consequence of Proposition \ref{asympt}.
\begin{corollary}
\label{positivity_ok}
Let $\G\in\mathbf{G}$ be either a noncompact graph with finitely many edges or a periodic graph.  There exists $\tilde{\mu} = \tilde{\mu}(\G, p)>0$ such that, if $\mu \in (0, \tilde{\mu}]$, then every solution $u \in H_\mu^1(\G)$ to \eqref{nlsG} (or \eqref{nlsG} in which $u^{p-1}$ is replaced by $\rho u^{p-1}$) satisfies 
$E(u, \G) > 0$ (respectively $E_{\rho}(u, \G) > 0$).
\end{corollary}

\section{Mountain pass structure and a priori estimates}
\label{sec:apriori}

In this section we introduce the mountain pass geometry that will be exploited to prove Theorems \ref{thm:exHalf}--\ref{thm:exper} and we derive topology-driven a priori estimates on the mountain pass level.

We start with the next lemma, analogous to those in \cite{BCJS_Non,CJS}, that provides the precise description of the mountain pass structure we will consider.
\begin{lemma}
	\label{lem:MPgeom}
	Let $\G\in\mathbf{G}$ be either a noncompact graph with finitely many edges or a periodic graph. Then, for every $p>6$ and every $\mu>0$, there exists $\delta>0$, depending on $\G,p,\mu$ but not on $\rho\in\left[1/2,1\right]$, such that
	\[
	c_\rho(\G):=\inf_{\gamma\in\Gamma}\max_{t\in[0,1]}E_\rho(\gamma(t),\G)>\max\left\{\sup_{u\in A_\delta}E_\rho(u,\G),\sup_{u\in B}E_\rho(u,\G)\right\}\qquad\forall \rho\in\left[1/2,1\right]\,,
	\]
	where $\Gamma:=\left\{\gamma\in C([0,1],H_\mu^1(\G))\,:\,\gamma(0)\in A_\delta,\; \gamma(1)\in B
	\right\}$ and  
	\[
	A_\delta:=\left\{u\in H_\mu^1(\G)\,:\,\|u'\|_2^2\leq\delta\right\}\,,\qquad B:=\left\{u\in H_\mu^1(\G)\,:\, E_{\frac12}(u,\G)<0\right\}\,.
	\]
\end{lemma}
\begin{proof}
	The proof works exactly as that of \cite[Lemma 3.1]{BCJS_Non}, additionally recalling here that, for every $\mu,k>0$, the set
	\[
	A_k:=\left\{u\in H_\mu^1(\G)\,:\,\|u'\|_2^2\leq k\right\}
	\]
	is nonempty whenever $\G\in\mathbf{G}$ is a noncompact graph with finitely many edges or a periodic graph because for all such graphs we already pointed out in \eqref{eq:lambdaG} that $\lambda_\G=0$.
\end{proof}

\begin{remark}
	\label{rem:pathR}
	On the real line it is well-known that 
	\[
	c_\rho(\R)=E_\rho(\phi_{\mu,\rho},\R)\qquad\forall \mu>0,\, \forall\rho\in\left[1/2,1\right], 
	\]
	where $\phi_{\mu,\rho}$ is as in Section \ref{sec:prel}. Furthermore, setting $a_p:=\left(\frac{p-2}{4}\right)^{\frac2{p-6}}$, exploiting \eqref{eq:relphi1} it is easy to see that there exists a (small) $\varepsilon_p>0$ such that the path $\overline{\gamma}:=(\overline{\gamma}_t)_{t\in[0,1]}\subset H_\mu^1(\R)$ given by
	\[
	\overline{\gamma}_t(x):=\sqrt{a_p t+\varepsilon_p}\,\phi_{\mu,\rho}\left((a_p t+\varepsilon_p)x\right)
	\]
	satisfies $\overline{\gamma}_0\in A_\delta$, $\overline{\gamma}_1\in B$ and $\displaystyle c_\rho(\R)=\max_{t\in[0,1]}E_\rho(\overline{\gamma}_t,\R)=E_\rho\left(\overline{\gamma}_{\frac{1-\varepsilon_p}{a_p}},\R\right)$. Observe also that, by \eqref{eq:relphi1}, \eqref{eq:Erhophi}, \eqref{eq:Emu}, 
	\[
	\|\overline{\gamma}_t'\|_2^2=\overline{C}\mu^{2\beta+1},\qquad\|\overline{\gamma}_t\|_p^p=\overline{C}'\mu^{2\beta+1}\,
	\]
	with $\overline{C},\overline{C}'>0$ depending on $p,\rho,t$ only, and being bounded both from above and away from zero, uniformly for $\rho\in\left[1/2,1\right]$ and $t\in[0,1]$.  Clearly, since for every $\mu$ and $\rho$ there results
	\[
	c_\rho(\R^+)= 2^{2 \beta}c_\rho(\R),
	\]
	an analogous path can be constructed at mass $\mu$ on the half-line simply taking the restriction to $\R^+$ of the path $\overline{\gamma}\subset H_{2\mu}^1(\R)$.
\end{remark}

The main result of this section is the following proposition, that establishes suitable a priori estimates on the level $c_\rho(\G)$.

\begin{proposition}
	\label{prop:level}
	For every $\G\in\mathbf{G}$ and every $p>6$, when $\mu \to 0$
	\[
	c_\rho(\G)\leq c_\rho(\R)+o(c_\rho(\R)).
	\]
	Moreover, if $\G$ has a pendant, 
	\[
	c_\rho(\G)\le c_\rho(\R^+)+o\left(c_\rho(\R^+)\right), 
	\]
	whereas if $\G$ has a signpost, then 
	\[
	c_\rho(\G)< c_\rho(\R).
	\]
	All relations are uniform for $\rho\in\left[1/2,1\right]$.
\end{proposition}
\begin{proof}
	We argue in three steps.
	
	\smallskip
	{\em Step 1: proof of $c_\rho(\G)\leq c_\rho(\R)+o(c_\rho(\R))$.} Let $\G\in\mathbf{G}$ and $p>6$ be fixed. Let $e\in\E_\G$ be any given edge of $\G$, and let $4\ell:=|e|$ be its length. To prove the first part of Proposition \ref{prop:level}, we will now show that, as soon as $\mu \to 0$, there exists a path $\gamma=(\gamma_t)_{t\in[0,1]}\subset H_\mu^1(\G)$ such that $\gamma_t$ is compactly supported on $e$ for every $t\in[0,1]$ and $\displaystyle\max_{t\in[0,1]}E_\rho(\gamma_t,\G)\leq c_\rho(\R)+o(c_\rho(\R))$ for every $\rho\in\left[1/2,1\right]$. 
	
	To this end, let $\overline{\gamma}\subset H_\mu^1(\R)$ be the path given in Remark \ref{rem:pathR} and define $v_t:e\to\R$ as
	\[
	v_t(x):=\begin{cases}
	\overline{\gamma}_t(x) & \text{if }x\in\left[-\ell,\ell\right]\\
	\overline{\gamma}_t(\ell)(2-x/\ell) & \text{if }x\in\left(\ell,2\ell\right]\\
	\overline{\gamma}_t(-\ell)(2+x/\ell) & \text{if }x\in\left[-2\ell,-\ell\right),
	\end{cases}
	\]
	where we tacitly identified the edge $e$ with the interval $[-2\ell,2\ell]$. Extending then $v_t$ to be identically zero on $\G\setminus e$, we define $\gamma_t\in H_\mu^1(\G)$ as
	\[
	\gamma_t:=\frac{\sqrt{\mu}}{\|v_t\|_2}\,v_t\,.
	\]
	Note that, by definition of $v_t$, Remark \ref{rem:pathR} and \eqref{eq:asphi}, 
	\begin{equation}
	\label{eq:vtL21}
	\begin{split}
	\|v_t\|_2^2=\|\overline{\gamma}_t\|_{L^2(-\ell,\ell)}^2+\frac{2\ell}3\overline{\gamma}_t(\ell)^2&\,=\mu+\frac{2\ell}3\overline{\gamma}_t(\ell)^2-2\|\overline{\gamma}_t\|_{L^2(\ell,+\infty)}^2\\
	&\,=\mu+C\mu^{2\alpha}e^{-2\tau c\mu^\beta\ell}+o\left(\mu^{2\alpha}e^{-2\tau c\mu^\beta\ell}\right)\qquad\text{as }\mu\to0\,,
	\end{split}
	\end{equation}
	for constants $C,c>0$ depending on $p,\rho,t$ only and both bounded from above and away from zero uniformly for $\rho\in\left[1/2,1\right]$ and $t\in[0,1]$, and $\tau:=\alpha/\beta$ (with $\alpha,\beta$ as in \eqref{eq:alfabeta}). Analogous computations lead to
	\[
	\begin{split}
	\|v_t'\|_2^2&\,=\|\overline{\gamma}_t'\|_2^2+C'\mu^{2\alpha+\beta}e^{-2\tau c\mu^\beta\ell}+o\left(\mu^{2\alpha+\beta}e^{-2\tau c\mu^\beta\ell}\right)=\|\overline{\gamma}_t'\|_2^2+o\left(\|\overline{\gamma}_t'\|_2^2\right)\\
	\|v_t\|_p^p&\,=\|\overline{\gamma}_t\|_p^p+C''\mu^{\alpha p}e^{-p\tau c\mu^\beta\ell}+o\left(\mu^{\alpha p}e^{-p\tau c\mu^\beta\ell}\right)=\|\overline{\gamma}_t\|_p^p++o\left(\|\overline{\gamma}_t\|_p^p\right)\qquad\text{as }\mu\to0\,,
	\end{split}
	\]
	again with $C',C''>0$ depending on $p,\rho,t$ only and both bounded from above and away from zero uniformly for $\rho\in\left[1/2,1\right]$ and $t\in[0,1]$ by Remark \ref{rem:pathR}. All in all, as soon as $\mu$ is small enough we obtain $\gamma_0\in A_\delta$, $\gamma_1\in B$ and
	\[
	E_\rho(\gamma_t,\G)=\frac{\mu}{\|v_t\|_2^2}\frac12\|v_t'\|_2^2-\left(\frac\mu{\|v_t\|_2^2}\right)^\frac p2 \frac1p\|v_t\|_p^p= E_\rho(\overline{\gamma}_t,\R)+o(E_\rho(\overline{\gamma}_t,\R))
	\]
	uniformly for $\rho\in\left[1/2,1\right]$ and $t\in[0,1]$, in turn yielding
	\[
	c_\rho(\G)\leq\max_{t\in[0,1]}E_\rho(\gamma_t,\G)=\max_{t\in[0,1]}\left(E_\rho(\overline{\gamma}_t,\R)+o\left(E_\rho(\overline{\gamma}_t,\R)\right)\right)=c_\rho(\R)+o(c_\rho(\R))
	\]
	for sufficiently small $\mu$.
	
	\smallskip
	{\em Step 2: proof of $c_\rho(\G) \leq c_\rho(\R^+)+o(c_\rho(\R^+))$ for graphs with a pendant.} This step works exactly as Step 1, with the only difference that we now construct a path $\gamma\subset H_\mu^1(\G)$ such that $\gamma_t$ is compactly supported on the pendant of $\G$ for every $t\in [0,1]$ and $\displaystyle\max_{t\in[0,1]}E_\rho(\gamma_t)\leq c_\rho(\R^+)+o(c_\rho(\R^+))<c_\rho(\R)$. By Remark \ref{rem:pathR}, such a path $\gamma$ can be obtained repeating verbatim the construction of Step 1 starting with the restriction to $\R^+$ of $\overline{\gamma}\subset H_{2\mu}^1(\R)$.
	
	\smallskip	
	{\em Step 3: proof of $c_\rho(\G)<c_\rho(\R)$ for graphs with a signpost.} The argument is again analogous to that of Step 1, suitably adapted to functions defined on the tadpole graph, i.e. the graph $\mathcal{T}_\ell$ given by a loop $\mathcal{L}$ of length $2\ell>0$ and a single half-line $\HH$. Note first that, for every $\mu$, $\rho$, it holds
	\begin{equation}
	\label{eq:levT}
	c_\rho(\mathcal{T}_\ell)<c_\rho(\R)\,.
	\end{equation}
	To see this, take again the path $\overline{\gamma}\subset H_\mu^1(\R)$ introduced in Remark \ref{rem:pathR} and define $\widetilde{\gamma}\subset H_\mu^1(\mathcal{T}_\ell)$ as
	\[
	\widetilde{\gamma}_t(x):=\begin{cases}
	\overline{\gamma}_t(x) & \text{if }x\in\left[-\ell,\ell\right]\\
	\overline{\gamma}_t\left(\frac{x+2\ell}2\right) & \text{if }x\in\R^+
	\end{cases}
	\]
	where, with a slight abuse of notation, we identified the loop $\mathcal{L}$ with the interval $[-\ell,\ell]$ and the half-line $\HH$ with $\R^+$. Straightforward computations show that, for every $t\in[0,1]$,
	\begin{equation}
	\label{eq:gammatilde}
	\begin{split}
	&\|\widetilde{\gamma}_t\|_2^2=\mu\,,\qquad\qquad\|\widetilde{\gamma}_t\|_{L^p(\mathcal{T}_\ell)}^p=\|\overline{\gamma}_t\|_{L^p(\R)}^p\,,\\
	&\|\widetilde{\gamma}_t'\|_{L^2(\mathcal{T}_\ell)}^2=\|\overline{\gamma}_t'\|_{L^2(-\ell,\ell)}^2+\frac14\|\overline{\gamma}_t'\|_{L^2(\R\setminus[-\ell,\ell])}^2=\|\overline{\gamma}_t'\|_2^2-\frac34\|\overline{\gamma}_t'\|_{L^2(\R\setminus[-\ell,\ell])}^2\,,
	\end{split}
	\end{equation}
	so that $\widetilde{\gamma}_0\in A_\delta$, $\widetilde{\gamma}_1\in B$ and $E_\rho(\widetilde{\gamma}_t,\mathcal{T}_\ell)<E_\rho(\overline{\gamma}_t,\R)$, in turn yielding \eqref{eq:levT}.
	
	To show then that the same inequality holds on $\G$ when it has a signpost, it is enough to adapt the construction in Step 1 to the path $\widetilde{\gamma}$. Denote by $\mathcal{P}$ the signpost of $\G$ and let $2\ell$ be the length of its loop and $2a$ be that of its bounded edge. For every $t\in[0,1]$, we define $v_t:\mathcal{P}\to\R$ as
	\[
	v_t(x):=\begin{cases}
	\widetilde{\gamma}_t(x) & \text{if }x\in\mathcal{L}\\
	\widetilde{\gamma}_t(x) & \text{if }x\in \left[0,a\right]\cap\HH\\
	\widetilde{\gamma}_t(a)\left(2-\frac xa\right) & \text{if }x\in[a,2a]\,,
	\end{cases}
	\]
	again with the slight abuse of notation to identify the loop of $\mathcal{P}$ with the loop $\mathcal{L}$ of the tadpole $\mathcal{T}_\ell$ and its bounded edge with the interval $[0,2a]$ on the half-line $\HH$ of $\mathcal{T}_\ell$. Extending $v_t$ to zero on the rest of $\G$, we then define $\gamma\subset H_\mu^1(\G)$ as $\displaystyle\gamma_t:=\frac{\sqrt{\mu}}{\|v_t\|_2}\,v_t$, for every $t\in[0,1]$. Arguing as in Step 1, by \eqref{eq:asphi} there results
	\begin{equation}
	\label{eq:vtL22}
	\begin{split}
	\|v_t\|_2^2=&\,\|\widetilde{\gamma}_t\|_{L^2(\mathcal{L}\cup(\HH\cap[0,a]))}^2	+\frac{\overline{\gamma}_t\left(\frac{a+2\ell}2\right)^2a}{3}=\mu-\|\widetilde{\gamma}_t\|_{L^2(\HH\cap(a,+\infty))}^2+\frac{\overline{\gamma}_t\left(\frac{a+2\ell}2\right)^2a}{3}\\
	=&\,\mu-\int_{a}^{+\infty}\left|\overline{\gamma}_t\left(\frac{x+2\ell}{2}\right)\right|^2\,dx+\frac{\overline{\gamma}_t\left(\frac{a+2\ell}2\right)^2a}{3}=\mu-2\int_{\frac{a+2\ell}{2}}^\infty\left|\overline{\gamma}_t(y)\right|^2\,dy+\frac{\overline{\gamma}_t\left(\frac{a+2\ell}2\right)^2a}{3}\\
	=&\,\mu+K\mu^{2\alpha}e^{-2\tau c\mu^\beta\left(\frac{a+2\ell}2\right)}+o\left(\mu^{2\alpha}e^{-2\tau c\mu^\beta\left(\frac{a+2\ell}2\right)}\right)\qquad\text{as }\mu\to0\,,
	\end{split}
	\end{equation}
	where $c,\tau$ are the same as in \eqref{eq:vtL21} and $K>0$ is a constant depending on $p,\rho,t$ only, bounded from above and away from zero uniformly for $\rho\in\left[1/2,1\right]$ and $t\in[0,1]$. Similarly, recalling also \eqref{eq:gammatilde}, we have
	\begin{equation}
	\label{eq:vt'}
	\begin{split}
	\|v_t'\|_2^2=&\,\|\widetilde{\gamma}_t'\|_{L^2\left(\mathcal{L}\cup(\HH\cap[0,a])\right)}^2+\frac{\overline{\gamma}_t\left(\frac{a+2\ell}{2}\right)^2}{a}=\|\widetilde{\gamma}_t\|_{L^2(\mathcal{T}_\ell)}^2-\|\widetilde{\gamma}_t'\|_{L^2(\HH\cap(a,+\infty))}^2+\frac{\overline{\gamma}_t\left(\frac{a+2\ell}{2}\right)^2}{a}\\
	=&\,\|\overline{\gamma}_t'\|_2^2-\frac34\|\overline{\gamma}_t'\|_{L^2(\R\setminus[-\ell,\ell])}^2-\|\widetilde{\gamma}_t'\|_{L^2(\HH\cap(a,+\infty))}^2+\frac{\overline{\gamma}_t\left(\frac{a+2\ell}{2}\right)^2}{a}\\
	=&\,\|\overline{\gamma}_t'\|_{2}^2-K'\mu^{2\alpha+\beta}e^{-2\tau c\mu^\beta\ell}+o\left(\mu^{2\alpha+\beta}e^{-2\tau c\mu^\beta\ell}\right)
	\end{split}
	\end{equation}
	and
	\begin{equation}
	\label{eq:vtLp}
	\begin{split}
	\|v_t\|_p^p=&\,\|\widetilde{\gamma}_t\|_{L^p(\mathcal{L}\cup(\HH\cap[0,a]))}^p+\frac{\left|\overline{\gamma}_t\left(\frac{a+2\ell}{2}\right)\right|^pa}{p+1}= \|\widetilde{\gamma}_t\|_{L^p(\mathcal{T}_\ell)}^p-\|\widetilde{\gamma}_t\|_{L^p(\HH\cap(a,+\infty))}^p+\frac{\left|\overline{\gamma}_t\left(\frac{a+2\ell}{2}\right)\right|^pa}{p+1}\\
	=&\,\|\overline{\gamma}_t\|_{L^p(\R)}^p+K''\mu^{\alpha p}e^{-p\tau c\mu^\beta\left(\frac{a+2\ell}{2}\right)}+o\left(e^{-p\tau c\mu^\beta\left(\frac{a+2\ell}{2}\right)}\right)\,,
	\end{split}
	\end{equation}
	for constants $K',K''>0$ depending on $p,\rho,t$ only and bounded from above and away from zero uniformly for $\rho\in\left[1/2,1\right]$ and $t\in[0,1]$.
	Coupling \eqref{eq:vt'} and \eqref{eq:vtLp} yields, since $p>2$ and $a>0$,
	\[
	E_\rho(v_t,\G)=E_\rho(\overline{\gamma}_t,\R)-K'\mu^{2\alpha+\beta}e^{-2\tau c\mu^\beta\ell}+o\left(\mu^{2\alpha+\beta}e^{-2\tau c\mu^\beta\ell}\right)
	\]
	as $\mu\to0$, which together with \eqref{eq:vtL22} (and recalling again Remark \ref{rem:pathR} and \eqref{eq:Emu}) gives
	\[
	\begin{split}
	E_\rho(\gamma_t,\G)&\,=\frac{\mu}{\|v_t\|_2^2}\frac12\|v_t'\|_2^2-\left(\frac\mu{\|v_t\|_2^2}\right)^\frac p2 \frac1p\|v_t\|_p^p=\left(1-O\left(\mu^{2\alpha}e^{-2\tau c\mu^\beta\left(\frac{a+2\ell}2\right)}\right)\right)E_\rho(v_t,\G)\\
	&\,= E_\rho(\overline{\gamma}_t,\R)-K'\mu^{2\alpha+\beta}e^{-2\tau c\mu^\beta\ell}E_\rho(\overline{\gamma}_t,\R)+o\left(\mu^{2\alpha+\beta}e^{-2\tau c\mu^\beta\ell}E_\rho(\overline{\gamma}_t,\R)\right)
	\end{split}
	\]
	as $\mu \to 0$. Since $K'$ is bounded away from zero uniformly for $\rho\in\left[1/2,1\right]$ and $t\in[0,1]$, by Remark \ref{rem:pathR} the previous identity entails $\displaystyle c_\rho(\G)\leq \max_{t\in[0,1]}E_\rho(\gamma_t,\G)<\max_{t\in[0,1]}E_\rho(\overline{\gamma}_t,\R)=c_\rho(\R)$ as soon as $\mu$ is small enough, and we conclude.
\end{proof}
We end this section with a more precise characterization of the mountain pass level $c_\rho(\G)$ under further topological conditions on the graph. To this end, we recall that a noncompact graph $\G\in\mathbf{G}$ with finitely many edges is said to satisfy assumption (H) if every $x\in\G$ lies on a trail (i.e. a path of adjacent edges of $\G$, each one run through exactly once) containing two half-lines. Assumption (H) has been introduced for the first time in \cite{AST}, but the formulation we reported here is taken from \cite{AST_Parma}.
\begin{proposition}
	\label{prop:H}
	Let $\G\in\mathbf{G}$ be a noncompact graphs with finitely many edges. Then 
	\begin{equation}
	\label{eq:cG=cR}
	c_\rho(\G)\leq c_\rho(\R)\qquad\forall\mu>0,\,\forall\rho\in\left[1/2,1\right].
	\end{equation}
	Furthermore, if $\G$ satisfies assumption (H), then equality holds in \eqref{eq:cG=cR}.
\end{proposition}
\begin{proof}
	The validity of \eqref{eq:cG=cR} is evident as soon as $\G$ has at least one half-line. To show that $c_\rho(\G)=c_\rho(\R)$ whenever $\G$ satisfies assumption (H) recall that, by standard properties of rearrangements on graphs (see e.g. \cite[Section 3]{AST}), if $u\in H^1(\G)$ is a positive function on $\G$ satisfying assumption (H), then its symmetric decreasing rearrangement on the line $\widehat{u}\in H^1(\R)$ is such that
	\begin{equation}
	\label{eq:riarr}
	\|u\|_r=\|\widehat{u}\|_r,\quad\forall r\geq1,\qquad\|u'\|_2\geq\|\widehat{u}'\|_2\,,
	\end{equation}
	By \eqref{eq:riarr} it follows that, if $\gamma\in C\left([0,1],H_\mu^1(\G)\right)$ is a path in $H_\mu^1(\G)$ such that $\gamma_0\in A_\alpha$ and $\gamma_1\in B$ on $\G$, then the symmetric decreasing rearrangement $\widehat{\gamma}_t$ of $\gamma_t$ provides a path $\widehat{\gamma}\in C\left([0,1], H_\mu^1(\R)\right)$  (the strong continuity in $H^1(\R)$ of the map $t\mapsto\widehat{\gamma}_t$ can be proved e.g. as in \cite{Coron}) such that $\widehat{\gamma}_0\in A_\alpha$ and $\widehat{\gamma}_1\in B$ on $\R$, and $E_\rho(\widehat{\gamma}_t,\R)\leq E_\rho(\gamma_t,\G)$ for every $t$. As this ensures that $c_\rho(\G)\geq c_\rho(\R)$, coupling with \eqref{eq:cG=cR} yields the desired identity.
\end{proof}

\section{Graphs with finitely many edges: proof of Theorem \ref{thm:exHalf}}
\label{sec:exhalf}

In this section we provide the proof of the first main existence result of the paper, i.e. the existence of positive solutions of mountain pass type on noncompact graphs with finitely many edges stated as in Theorem \ref{thm:exHalf}. 

Before proving Theorem \ref{thm:exHalf}, we establish a useful alternative for bounded Palais-Smale sequences of $E_\rho$ with small mass.
\begin{lemma}
	\label{lem:dicHalf}
	Let $\G\in\mathbf{G}$ be a noncompact graph with finitely many edges and $p>6$. There exists $\overline{\mu}=\overline{\mu}(p,\G)>0$ such that, for every $\mu\in\left(0,\overline{\mu}\right]$ and every $\displaystyle\rho\in\left[1/2,1\right]$, if $(u_n)_n\subset H_\mu^1(\G)$ satisfies $u_n\geq0$ on $\G$ and, as $n\to+\infty$,
	\begin{itemize}
		\item[(a)] $\displaystyle E_\rho(u_n,\G)\to c_\rho(\G)$,
		\item[(b)] $\displaystyle E_\rho'(u_n,\G)+\lambda_n u_n\to0$ in the dual of $H^1(\G)$, with $\displaystyle \lambda_n:=-\frac{E_\rho'(u_n,\G)[u_n]}{\mu}=\frac{\rho\|u_n\|_p^p-\|u_n'\|_2^2}{\mu}$ for every $n$, and
		\item[(c)] $\displaystyle u_n\rightharpoonup u$ in $H^1(\G)$,
	\end{itemize}
	then either $u\equiv 0$ on $\G$ or $u$ is a critical point of $E_\rho$ constrained to $H_\mu^1(\G)$ at level $c_\rho(\G)$.
\end{lemma}
\begin{proof}
	Since $(u_n)_n$ is bounded in $H^1(\G)$ by (c), so is $(\lambda_n)_n \subset \R$. Hence, up to subsequences, $\lambda_n\to\lambda$ as $n\to+\infty$, for some $\lambda\in\R$.
	
	By weak lower semicontinuity, $m:=\|u\|_2^2\in[0,\mu]$. If $m=0$, then $u\equiv0$. Conversely, if $m=\mu$, then $u\in H_\mu^1(\G)$, $u_n\to u$ in $L^2(\G)$ and thus in $L^p(\G)$, so that by (b) and $\lambda_n\to\lambda$ it follows
	\[
	\|u_n'-u'\|_2^2+\lambda\|u_n-u\|_2^2=o(1)\qquad\text{as }n\to+\infty\,,
	\]
	that is $u_n\to u$ in $H^1(\G)$, and $u$ is a critical point of $E_\rho$ constrained to $H_\mu^1(\G)$ at level $c_\rho(\G)$. Hence, to prove the lemma it is enough to rule out the case $m\in(0,\mu)$.
	
	To this end, we argue by contradiction. Assume thus $0<m<\mu$, so that $u\not\equiv0$ on $\G$ satisfies
	\begin{equation}
	\label{eq:NLSu}
	\begin{cases}
	u''+\rho u^{p-1}=\lambda u & \text{on every }e\in\E_\G\\
	u\geq 0 &\text{on }\G\\
	\sum_{e\succ \vv}\frac{du}{dx_e}(\vv)=0 & \text{for every }\vv\in\V_\G\,.
	\end{cases}
	\end{equation} 
	By standard arguments, $u>0$ on $\G$, that is $u\in S_{m,\rho}$ (where $S_{m,\rho}$ is the set defined in Remark \ref{rem:estErho}). Thus we have from Lemma \ref{Properties_L}(c)  that $\lambda >0$.  By Remark \ref{rem:estErho}, see also Corollary \ref{positivity_ok},
	there exist $\overline{\mu}>0$  (depending on $p$ and $\G$ only), such that for every $\mu\in\left(0,\overline{\mu}\right)$ and every $\displaystyle\rho\in\left[1/2,1\right]$, there results
	\begin{equation}
	\label{eq:Eu>0}
	E_\rho(u,\G) \ge \frac{1}{2} E_\rho(\phi_{2m, \rho}, \R^+)\,.
	\end{equation}
	Set then $v_n:=u_n-u$. By weak convergence of $u_n$ to $u$ in $H^1(\G)$, as $n \to \infty$ we have
	\begin{equation}
	\label{eq:vnvn'L2}
	\|v_n\|_2^2=\mu-m+o(1),\qquad\|v_n'\|_2^2=\|u_n'\|_2^2-\|u'\|_2^2+o(1)\,,
	\end{equation}
	so that by Brezis-Lieb lemma \cite{BL} and (a)
	\begin{equation}
	\label{eq:Evn}
	E_\rho(v_n,\G)=E_\rho(u_n,\G)-E_\rho(u,\G)+o(1)=c_\rho(\G)-E_\rho(u,\G)+o(1)\qquad\text{as }n\to+\infty\,.
	\end{equation}
	Moreover, by (b), (c), \eqref{eq:NLSu} and $\lambda_n\to\lambda$, it is easy to see that $v_n\to0$ in $H^1(\K)$, where $\K$ denotes as usual the compact core of $\G$. This can be easily seen for instance by taking the difference between (b) and \eqref{eq:NLSu} both tested with $\varphi_n:=(u_n-u)\eta$, where $\eta$ is compactly supported on a fixed neighbourhood of $\K$ and $\eta\equiv1$ on $\K$. Hence, up to negligible corrections in $H^1(\G)$ that do not affect the validity of \eqref{eq:vnvn'L2} and \eqref{eq:Evn}, we can take $v_n$ to be identically zero on $\K$ for every $n$. Since $\displaystyle\G\setminus\K=\bigcup_{i=1}^N\HH_i$, where $\HH_i,\,i=1,\dots,N$, are all the half-lines of $\G$, this means that we can think of $v_n$ as the sum of its restrictions $v_{n,i}$ to the half-lines $\HH_i$, $i=1,\dots,N$. As $v_{n,i}$ vanishes at the vertex of $\HH_i$, we can embed $\HH_i$ in $\R$ and extend each $v_{n,i}$ by $0$. Up to translations, there is no loss of generality in further assuming that each $v_{n,i}$ satisfies $\|v_{n,i}\|_\infty=v_{n,i}(0)$. Moreover, standard arguments (see e.g. \cite{CZR}) show that, for every $i=1,\dots, N$, as $n\to+\infty$,
	\begin{equation}
	\label{eq:E'v_nR}
	\int_\R v_{n,i}'\varphi'\,dx-\rho\int_\R v_{n,i}^{p-1}\varphi\,dx+\lambda\int_\R v_{n,i}\varphi\,dx=o(1)\|\varphi\|_{H^1(\R)}\qquad\forall \varphi\in H^1(\R)\,.
	\end{equation}
	Observe that, together with \eqref{eq:vnvn'L2} and $\lambda >0$, this entails that (up to subsequences) $\|v_{n,i}\|_\infty\not\to0$ as $n\to+\infty$ at least for one value of $i$. Indeed, otherwise $\|v_{n,i}\|_p^p \to0$ as $n\to+\infty$ contradicting \eqref{eq:E'v_nR}.
	
	All in all, \eqref{eq:vnvn'L2} and \eqref{eq:E'v_nR} imply that there exists $k\in\N$, $k\geq1$, such that
	\[
	E_\rho(v_n,\G)=\sum_{i=1}^NE_\rho(v_{n,i},\R)=k E_\rho\left(\phi_{\frac{\mu-m}{k},\rho},\R\right)+o(1)
	\]
	as soon as $n$ is large enough. Coupling with \eqref{eq:Eu>0} and \eqref{eq:Evn} 	
	then yields (recalling \eqref{eq:Emu})
	\begin{align*}
	c_\rho(\G) &\geq k E_\rho\left(\phi_{\frac{\mu-m}{k},\rho},\R\right)+ E_\rho(u,\G) +
	o(1) \geq E_\rho\left(\phi_{\mu-m,\rho},\R\right)+ \frac{1}{2} E_\rho(\phi_{2m,\rho},\R^+) +o(1)  \\
	&= \theta_p \rho^{\frac{4}{6-p}}\left((\mu-m)^{2\beta +1} +  2^{2\beta -1} m^{2\beta +1}\right) +o(1) \ge (1-\tau)^{2\beta+1}E_\rho(\phi_{\mu,\rho},\R)+ o(1)
	\end{align*}
	where $(1-\tau)^{2\beta+1}>1$ since $2\beta + 1 <0$, and $\tau:=\left(2 (\frac{1}{2})^{-2\beta}+1\right)^{-1}$ is such that $\tau\mu$ is the minimum point of the function $m\mapsto (\mu-m)^{2\beta +1} +  2^{2\beta -1} m^{2\beta +1}$ on $(0,\mu)$. Note that $\tau$ depends on $p$ only.\\
	However, by Proposition \ref{prop:level} (and recalling that $\tau<1$ and $2\beta+1<0$), up to a possible reduction of  the threshold $\overline{\mu}>0$, there results $c_\rho(\G)\leq (1-\tau/2)^{2\beta+1}c_\rho(\R)=(1-\tau/2)^{2\beta+1} E_\rho(\phi_{\mu,\rho},\R)$ for every $\mu\in(0,\overline{\mu}]$ and every $\displaystyle\rho\in\left[1/2,1\right]$, contradicting the preceding inequality and  completing the proof.
\end{proof} 

\begin{proof}[Proof of Theorem \ref{thm:exHalf}]
	The proof is divided in two steps: in the first one, we prove existence of positive critical points of $E_\rho$ constrained to $H_\mu^1(\G)$ for almost every $\displaystyle\rho\in\left[1/2,1\right]$, in the second one we prove that (a subsequence of) these critical points converges, as $\rho\to1^-$, to a positive critical point of $E$ in $H_\mu^1(\G)$.
	
	\smallskip
	{\em Step 1.} For every fixed $\mu>0$, by Lemma \ref{lem:MPgeom} and Theorem \ref{thm:PS} it follows that, for almost every $\displaystyle\rho\in\left[1/2,1\right]$, there exists a bounded sequence $(u_n)_n\subset H_\mu^1(\G)$ satisfying properties (i)--(iv) of Theorem \ref{thm:PS}. Furthermore, as pointed out in \cite[Remark 1.4]{BCJS_TAMS}, it is not restrictive to assume $u_n\geq0$ on $\G$ for every $n$. Set
	\begin{equation}
	\label{eq:lambdan}
	\lambda_n:=-\frac{E_\rho'(u_n,\G)[u_n]}{\|u_n\|_2^2}=\frac{\rho\|u_n\|_p^p-\|u_n'\|_2^2}{\mu}\,.
	\end{equation}
	Since $(u_n)_n$ is bounded in $H^1(\G)$ by property (iii) of Theorem \ref{thm:PS}, $(\lambda_n)_n$ is bounded too, so that as $n\to+\infty$ up to subsequences $\lambda_n\to\lambda_\rho$, for some $\lambda_\rho\in\R$, and $u_n\rightharpoonup u_\rho$ in $H^1(\G)$, for some $u_\rho\in H^1(\G)$ satisfying
	\begin{equation}
	\label{eq:NLSurho}
	\begin{cases}
	u_\rho''+u_\rho^{p-1}=\lambda_\rho u_\rho & \text{on every }e\in\E_\G\\
	u_\rho\geq0 &\text{on }\G\\
	\sum_{e\succ \vv}\frac{du_\rho}{dx_e}(\vv)=0 & \text{for every }\vv\in\V_\G\,.
	\end{cases}
	\end{equation}
	Since $\G$ is a graph with at least one half-line and $u_n$ satisfies properties (iv) of Theorem \ref{thm:PS}, by Lemma \ref{lem:dim3} and \cite[Lemma 3.2]{BCJS_Non} we obtain $\lambda_\rho\geq0$.
	
	We now show that $u_\rho\not\equiv0$ on $\G$. Assume by contradiction that this is the case, so that $u_n\rightharpoonup 0$ in $H^1(\G)$ and $u_n\to0$ in $L_{loc}^\infty(\G)$ as $n\to+\infty$. Note that the convergence of $u_n$ to 0 cannot be in $L^\infty(\G)$, because then $u_n\in H_\mu^1(\G)$ would imply $u_n\to0$ in $L^p(\G)$, which together with \eqref{eq:lambdan} and $\liminf \lambda_n\geq0$ would yield $E_\rho(u_n,\G)\to0$, that is impossible since $E_\rho(u_n,\G)\to c_\rho(\G)>0$. Hence, $u_n\to0$ in $L_{loc}^\infty(\G)$ but not in $L^\infty(\G)$. Whenever the mass is smaller than a threshold independent of $\rho$, arguing as in the proof of Lemma \ref{lem:dicHalf} it is then easy to see that, denoting by $\HH_i$, $i=1,\dots,N$, the half-lines of $\G$, as $n\to+\infty$ there results
	\[
	\begin{split}
	&\|u_n\|_{L^2(\bigcup_{i=1}^N\HH_i)}^2=\mu+o(1)\\
	&\int_{\HH_i}u_n'\varphi'\,dx-\int_{\HH_i}u_n^{p-1}\varphi\,dx+\lambda\int_{\HH_i}u_n\varphi\,dx=o(1)\|\varphi\|_{H^1(\HH_i)},\qquad\forall i=1,\dots,N\,,
	\end{split}
	\]
	and, up to negligible corrections in $H^1(\G)$, $u_n$ can be taken to be identically zero on $\K$. Since $\|u_n\|_\infty$ is uniformly bounded away from 0, it then follows that there exists $k\in\N$, $k\geq1$, such that
	\begin{equation}
	\label{eq:ass}
	E_\rho(u_n,\G)=k E_\rho\left(\phi_{\frac\mu k,\rho},\R\right)+o(1)\geq E_\rho(\phi_{\mu,\rho},\R)+o(1)=c_\rho(\R)+o(1)\qquad\text{as }n\to+\infty\,, 
	\end{equation}
	the first inequality following by \eqref{eq:Emu}. However, since by assumption $\G$ contains either a pendant or a signpost, Proposition \ref{prop:level} ensures that $E_\rho(u_n,\G)=c_\rho(\G)+o(1)<c_\rho(\R)$ as $\mu \to 0$  (uniformly in $\rho$), contradicting \eqref{eq:ass}. Hence, $u_{\rho}\not\equiv0$ on $\G$.
	
	Up to possibly lowering the threshold on the mass, by Lemma \ref{lem:dicHalf} we then have that $u_\rho$ is a critical point of $E_\rho$ constrained to $H_\mu^1(\G)$ at level $c_\rho(\G)$. Moreover, since $u_\rho$ satisfies \eqref{eq:NLSurho} and $\G$ is a noncompact graph with finitely many edges, we have $\lambda_\rho>0$ and $u_\rho>0$ on $\G$. 
	
	\smallskip
	{\em Step 2.} Since the previous argument works for every $\mu\in(0,\mu_{p,\G}]$ and almost every $\displaystyle\rho\in\left[1/2,1\right]$, for some value $\mu_{p,\G}>0$ depending on $p$ and $\G$ only, by Step 1 there exists $u_\rho>0$ critical point of $E_\rho$ constrained to $H_\mu^1(\G)$ at level $c_\rho(\G)$, for every $\mu\in(0,\mu_{p,\G}]$ and for almost every $\displaystyle\rho\in\left[1/2,1\right]$. We then fix any $\mu$ in this interval and consider the limit $\rho\to1^-$ (along a subsequence of $\rho$'s, not relabebed). Recall first that $u_\rho$ satisfies \eqref{eq:NLSurho} with $\lambda_\rho>0$. Since $\|u_\rho\|_2^2=\mu$ for every $\rho$, arguing exactly as in the proof of Proposition \ref{asympt} proves that $(\lambda_\rho)_\rho$ is bounded as $\rho\to1^-$.  By \eqref{eq:NLSurho}, the boundedness of $\lambda_\rho$ and that of $E_\rho(u_\rho,\G)=c_\rho(\G)$, we obtain that $(u_\rho)_\rho$ is bounded in $H_\mu^1(\G)$ as $\rho\to1^-$. All in all, this implies that, as $\rho\to 1$, $(u_\rho)_\rho\subset H_\mu^1(\G)$ satisfies $E(u_\rho)\to c(\G)+o(1)$. Indeed, the continuity of the map $\rho\mapsto c_\rho(\G)$ as $\rho\to1^-$ follows repeating the argument of \cite[Lemma 2.3]{J}. Furthemore,
	\[
	E'(u_\rho,\G)+\lambda_{\rho} u_\rho\to0\qquad\text{in the dual of }H^1(\G)
	\] 
	where, up to subsequences, $\lambda_\rho\to\lambda\geq0$ as $\rho\to1^-$. Therefore, repeating verbatim the argument of Step 1 yields a positive critical point of $E$ in $H_\mu^1(\G)$ at level $c(\G)$, i.e a solution of \eqref{nlsG} as desired.
\end{proof}

\section{Periodic graphs: proof of Theorem \ref{thm:exper}}
\label{sec:exper}
In this section we prove our existence result for mountain pass solutions on periodic graphs. The line of the argument is essentially the same as the one exploited in the previous section for graphs with half-lines, but the periodicity of the graph will now be crucial to recover the necessary compactness. The first result in this direction is indeed the analogue of Lemma \ref{lem:dicHalf}.
\begin{lemma}
	\label{lem:dicPer}
	Let $\G\in\mathbf{G}$ be a periodic graph and $p>6$. There exists $\overline{\mu}=\overline{\mu}(p,\G)>0$ such that, for every $\mu\in\left(0,\overline{\mu}\right]$ and every $\displaystyle\rho\in\left[1/2,1\right]$, if $(u_n)_n\subset H_\mu^1(\G)$ satisfies $u_n\geq0$ on $\G$ and, as $n\to+\infty$,
	\begin{itemize}
		\item[(a)] $\displaystyle E_\rho(u_n,\G)\to c_\rho(\G)$,
		\item[(b)] $\displaystyle E_\rho'(u_n,\G)+\lambda_n u_n\to0$ in the dual of $H^1(\G)$, with $\displaystyle \lambda_n:=-\frac{E_\rho'(u_n, \G)[u_n]}{\mu}$ for every $n$, and
		\item[(c)] $\displaystyle u_n\rightharpoonup u$ in $H^1(\G)$,
	\end{itemize}
	then either $u\equiv 0$ on $\G$ or $u$ is a critical point of $E_\rho$ constrained to $H_\mu^1(\G)$ at level $c_\rho(\G)$.
\end{lemma}
\begin{proof}
	Since the proof is very similar to that of Lemma \ref{lem:dicHalf}, we only highlight the main differences. 
	
	Let us consider first the case of a periodic graph $\G$ with no pendant. If we assume $m:=\|u\|_2^2\in(0,\mu)$, arguing exactly as above we obtain that $u\in S_{m,\rho}$ solves \eqref{eq:NLSu} for some  $\lambda \in \R$, where up to subsequences $\lambda_n\to\lambda$ as $n\to+\infty$. Note that by Lemma \ref{Properties_L}(b), $\lambda \geq 0$ and also, assuming $\mu >0$ sufficiently small, $\lambda >0$ by Lemma \ref{Properties_L}(c).  Moreover, exploiting the periodicity of $\G$, we can consider $u_n-u$ and suitably translate it on $\G$ to obtain a function $v_n$ such that, as $n\to+\infty$,
	\[
		\begin{split}
		&\|v_n\|_2^2=\mu-m+o(1)\\
		&E_\rho(v_n,\G)=c_\rho(\G)-E_\rho(u,\G)+o(1)\\
		&E_\rho'(v_n,\G)+\lambda v_n\to0\text{ in the dual of }H^1(\G)
		\end{split}
	\]
	and, for every $n$, $v_n$ attains its $L^\infty$ norm in a given compact subset of $\G$ independent of $n$. Since $(v_n)_n$ is bounded in $H^1(\G)$ by definition,  the argument used for $u_n$ implies that $v_n\rightharpoonup w$ in $H^1(\G)$, for some $w$ satisfying again \eqref{eq:NLSu} with the same $\lambda$ as $u$. Iterating this procedure, we eventually end up with finitely many functions $w_k$, $k=1,\dots, N$, each one solving \eqref{eq:NLSu} with the same $\lambda>0$ (here $w_1=u$, $w_2=w$ as above) and such that
	\[
	\mu=\|u_n\|_2^2=\sum_{k=1}^N\|w_k\|_2^2,\qquad c_\rho(\G)=E_\rho(u_n,\G)=\sum_{k=1}^NE_\rho(w_k,\G)+o(1)
	\]
	as $n \to +\infty$. Note however that $N\geq2$ since by assumption $\|w_1\|_2^2=m\in(0,\mu)$, so that in particular there exists $\overline{k}\in\left\{1,\dots,N\right\}$ for which $\displaystyle\|w_{\overline{k}}\|_2^2\leq\frac\mu2$. Remark \ref{rem:estErho} and Proposition \ref{prop:level} then ensure the existence of a value $\overline{\mu}>0$, depending only on $p$ and $\G$, such that  for every $\displaystyle\rho\in\left[1/2,1\right]$ and $\mu \in (0, \overline{\mu}]$, 
	\[
	E_\rho(\phi_{\mu,\rho},\R)+o\left(E_\rho(\phi_{\mu,\rho},\R)\right)\geq c_\rho(\G)>E_\rho(w_{\overline{k}},\G)\geq E_\rho\left(\phi_{\frac\mu2,\rho},\R\right)+o\left(E_\rho\left(\phi_{\frac\mu2,\rho},\R\right)\right)\,,
	\]
	i.e. a contradiction in view of \eqref{eq:Emu}.
	
	If $\G$ has a pendant, the proof is exactly the same, making use of the corresponding estimates in Remark \ref{rem:estErho} and Proposition \ref{prop:level}.
\end{proof} 
We are now almost ready to prove Theorem \ref{thm:exper}, for which we only need the following preliminary lemma, that is the counterpart of \cite[lemma 3.2]{BCJS_Non} in the context of periodic graphs.

\begin{lemma}
	\label{lem:Perdim3}
	Let $\G\in\mathbf{G}$ be a periodic graph. For every $\lambda<0$ there exists a subspace $Y\subset H^1(\G)$ such that $\text{\normalfont dim}\,Y=3$ and
	\begin{equation}
	\label{eq:ineqY}
	\|w'\|_2^2+\lambda\|w\|_2^2\leq\frac\lambda2\|w\|_{H^1(\G)}^2\qquad\forall w\in Y.
	\end{equation}
\end{lemma}
\begin{proof}
	It is a direct consequence of the fact that $\lambda_\G=0$ (with $\lambda_\G$ given by \eqref{eq:lambdaG}) for every periodic graph $\G$. Indeed, from this it follows that, for every $\lambda<0$, there exists $w_0\in H^1(\G)$ with compact support such that
	\[
	\|w_0'\|_2^2\leq \frac{-\lambda}{2-\lambda}\|w_0\|_2^2\,,
	\]
	i.e. $w_0$ satisfies \eqref{eq:ineqY}. Exploiting the periodicity of the graph, it is then straightforward to take $w_1,w_2\in H^1(\G)$ to be two different copies of $w_0$, suitably translated on $\G$ so that $w_0, w_1, w_2$  have disjoint supports. The proof is then completed simply taking $Y=\text{\normalfont span}\left\{w_0,w_1,w_2\right\}$.
\end{proof}
\begin{proof}[Proof of Theorem \ref{thm:exper}]
	The proof follows exactly the same line as that of Theorem \ref{thm:exHalf}, replacing whenever needed Lemma \ref{lem:dicHalf} with Lemma \ref{lem:dicPer}, and \cite[Lemma 3.2]{BCJS_Non} with Lemma \ref{lem:Perdim3}. The only significant difference arises when one shows that the weak limit $u_\rho$ of the Palais-Smale sequence $(u_n)_n\subset H_\mu^1(\G)$ for $E_\rho$ at level $c_\rho(\G)$ is not identically equal to 0 on $\G$. In the case of a periodic graph, it is even easier to perform this passage. Indeed, exploiting the periodicity of $\G$, there is no loss of generality in taking from the very beginning each $u_n$ to attain its $L^\infty$ norm in a given compact subset of $\G$ independent of $n$. Like this, if we assume by contradiction that $u_n\rightharpoonup0$ in $H^1(\G)$ as $n\to+\infty$, then $u_n\to0$ in $L^p(\G)$, in turn entailing (just as in the proof of Theorem \ref{thm:exHalf}) $\|u_n'\|_2^2\to0$ and $E_\rho(u_n,\G)\to0$ as $n\to+\infty$, that is impossible since $E_\rho(u_n,\G)\to c_\rho(\G)>0$ by assumption.
\end{proof}

\section{Proof of Theorems \ref{thm:neg}--\ref{thm:prop_p}}
\label{sec:thmneg}

This section is devoted to the proof of Theorems \ref{thm:neg}--\ref{thm:prop_p}, i.e. to show the existence of solutions of \eqref{nlsG} with negative energy on noncompact graphs $\G\in\mathbf{G}$ with finitely many edges and suitable topological assumptions when the mass is large enough.

We begin with the proof of Theorem \ref{thm:neg}. Since we will prove the theorem with a perturbative argument around $p=6$, let us briefly recall some previous results at this critical exponent. First, when $p=6$ it is well-known that the function $\phi_{\mu,1}$ as in Section \ref{sec:prel} solves the minimization problem
\[
\inf_{u\in H_\mu^1(\R)}E(u,\R)
\]
if and only if $\displaystyle\mu=\mu_\R:=\frac{\sqrt{3}\pi}2$. In this case, $E(\phi_{\mu_\R,1},\R)=0$.  Moreover, for every $\lambda>0$, the function $\varphi_\lambda(x):=\sqrt{\lambda}\phi_{\mu_\R,1}(\lambda x)$ satisfies $\varphi_\lambda\in H_{\mu_\R}^1(\R)$, $E(\varphi_\lambda,\R)=0$ and 
\begin{equation}
\label{eq:Ephibar}
\frac{\lambda}2\mu_\R=J_{\lambda,6}(\varphi_\lambda,\R)=\inf_{v\in\NN_{\lambda,6}(\R)}J_{\lambda,6}(v,\R)\,
\end{equation}
where 
\begin{equation}
\label{eq:JR}
J_{\lambda,p}(v,\R):=\frac12\|v'\|_2^2+\frac{\lambda}2\|v\|_2^2-\frac1p\|v\|_p^p
\end{equation}
and 
\begin{equation}
\label{eq:NR}
\NN_{\lambda,p}(\R):=\left\{v\in H^1(\R)\,:\,\|v'\|_2^2+\lambda\|v\|_2^2=\|v\|_p^p\right\}\,.
\end{equation}
Conversely, when $p=6$ and $\G$ has finitely many edges, no pendant and exactly one half-line, it has been shown in \cite[Theorem 3.3]{AST_CMP} that the minimization problem $\displaystyle \inf_{u\in H_\mu^1(\G)}E(u,\G)$ has a solution if and only if $\displaystyle\mu\in\left(\mu_\R/2,\mu_\R\right]$. Furthermore, for every $\displaystyle\mu\in\left(\mu_\R/2,\mu_\R\right]$, if $u\in H_\mu^1(\G)$ is such that $\displaystyle E(u,\G)= \inf_{u\in H_\mu^1(\G)}E(u,\G)$, then $E(u,\G)<0$. 

The following preliminary lemma moves from these considerations to establish existence of solutions to the problem
\[
\inf_{v\in\NN_{\lambda,p}(\G)}J_{\lambda,p}(v,\G)
\]
for $p$ slightly above $6$ (where $J_{\lambda,p}(v,\G)$, $\NN_{\lambda,p}(\G)$ are defined analogously to \eqref{eq:JR}, \eqref{eq:NR}).
\begin{lemma}
	\label{lem:J}
	Let $\G\in\mathbf{G}$ be a noncompact graph with finitely many edges, no pendant and exactly one half-line. Then, there exist $\overline{\lambda}>0$ and $\delta>0$ (depending on $\G$) such that $\displaystyle\inf_{v\in\NN_{\overline{\lambda},p}(\G)}J_{\overline{\lambda},p}(v,\G)$ is attained, for every $p\in[6,6+\delta)$.
\end{lemma}
\begin{proof}
	Recall first that, by \cite{DDGST}, if $\G\in\mathbf{G}$ is a noncompact graph with finitely many edges, for every $p>2$ and $\lambda>0$ the minimization problem $\displaystyle \inf_{v\in\NN_{\lambda,p}(\G)}J_{\lambda,p}(v,\G)$ admits a solution if
	\begin{equation}
	\label{eq:JG<JR}
	\inf_{v\in\NN_{\lambda,p}(\G)}J_{\lambda,p}(v,\G)<\inf_{v\in\NN_{\lambda,p}(\R)}J_{\lambda,p}(v,\R)\,.
	\end{equation}
	Hence, to prove the lemma it is enough to show that \eqref{eq:JG<JR} holds at a certain value of $\lambda$ for every $p$ sufficiently close to $6$, whenever $\G$ has finitely many edges, exactly one half-line and no pendant. Moreover, since at fixed $\lambda>0$ the map $\displaystyle p\mapsto\inf_{v\in\NN_{\lambda,p}(\G)}J_{\lambda,p}(v,\G)$ is easily proved to be continuous, we are left to prove the validity of  \eqref{eq:JG<JR} at $p=6$.
	
	Let then $p=6$. In this case, by the aforementioned result of \cite{AST_CMP}, there exists $\overline{u}\in H_{\mu_\R}^1(\G)$ such that
	\begin{equation}
	\label{eq:Eubar}
	E(\overline{u},\G)=\inf_{u\in H_{\mu_\R}^1(\G)}E(u,\G)<0\,.
	\end{equation}
	By \cite[Theorem 1.3]{DST_MA}, it then follows that
	\[
	J_{\overline{\lambda},6}(\overline{u},\G)=\inf_{v\in\NN_{\overline{\lambda},6}(\G)}J_{\overline{\lambda},6}(v,\G), \quad \text{with} \quad \overline{\lambda}:=\frac{\|\overline{u}\|_6^6-\|\overline{u}'\|_2^2}{\mu_\R}\,.
	\] 
	 By \eqref{eq:Ephibar} and \eqref{eq:Eubar}, we then have
	\[
	\inf_{v\in\NN_{\overline{\lambda},6}(\G)}J_{\overline{\lambda},6}(v,\G)=E(\overline{u},\G)+\frac{\overline{\lambda}}{2}\|\overline{u}\|_2^2<\frac{\overline{\lambda}}{2}\mu_\R=\inf_{v\in\NN_{\overline{\lambda},6}(\R)}J_{\overline{\lambda},6}(v,\R)\,,
	\]
	and we conclude.
\end{proof}

\begin{proof}[Proof of Theorem \ref{thm:neg}]
	Since it is a standard fact that solutions of the problem $\displaystyle\inf_{v\in\NN_\lambda(\G)}J_{\lambda,p}(v,\G)$ satisfy \eqref{nlsG}, in view of Lemma \ref{lem:J} to complete the proof of Theorem \ref{thm:neg} it is enough to show that, as soon as $p$ is sufficiently close to 6, functions $u\in\NN_{\overline{\lambda},p}(\G)$ for which $\displaystyle J_{\overline{\lambda},p}(u,\G)=\inf_{v\in\NN_{\overline{\lambda},p}(\G)}J_{\overline{\lambda},p}(v,\G)$ satisfy $E(u,\G)<0$ (where $\overline{\lambda}$ is the number given by Lemma \ref{lem:J}). 
	
	To this end, we argue by contradiction. Assume that there exist exponents $p_n\to 6^+$ and functions $u_n\in\NN_{\overline{\lambda},p_n}(\G)$ such that $\displaystyle J_{\overline{\lambda},p_n}(u_n,\G)=\inf_{v\in\NN_{\overline{\lambda},p_n}(\G)}J_{\overline{\lambda},p_n}(v,\G)$ and $E(u_n,\G)\geq0$. Since $p_n$ is bounded and the proof of Lemma \ref{lem:J} shows that $\displaystyle \inf_{v\in\NN_{\overline{\lambda},p_n}(\G)}J_{\overline{\lambda},p_n}(v,\G)<\inf_{v\in\NN_{\overline{\lambda},p_n}(\R)}J_{\overline{\lambda},p_n}(v,\R)$, it follows that $(u_n)_n$ is bounded in $H^1(\G)$. Hence, up to subsequences  $u_n\rightharpoonup u$ in $H^1(\G)$ and $u_n\to u$ in $L_{loc}^\infty(\G)$ as $n\to+\infty$, for some $u\in H^1(\G)$. 
	
	Note first that $u\not\equiv0$ on $\G$. Indeed, if it were $u\equiv0$ on $\G$, we would have $u_n\to0$ in $L^\infty(\K)$. Therefore, denoting by $\HH$ the unique half-line of $\G$, it would follow
	\begin{equation}
	\label{eq:assurdo}
	\begin{split}
	\inf_{v\in\NN_{\overline{\lambda},p_n}(\G)}J_{\overline{\lambda},p_n}(v,\G)=J_{\overline{\lambda},p_n}(u_n,\G)=J_{\overline{\lambda},p_n}(u_n,\HH)+o(1)\geq \inf_{v\in \NN_{\overline{\lambda},p_n}(\R)} J_{\overline{\lambda},p_n}(v,\R)+o(1)
	\end{split}
	\end{equation}
	which is impossible by the proof of Lemma \ref{lem:J}.
	
	Arguing in a similar way, we also observe that $u_n$ converges strongly in $H^1(\HH)$ to $u$. Indeed, $u_n$ being a solution of \eqref{nlsG} with $\lambda=\overline{\lambda}$ implies that $u_n(x)=\varphi_{\overline{\lambda},p_n}(x+a_n)$ for every $x\in\HH$, where $\varphi_{\overline{\lambda},p_n}$ is the unique $H^1$ solution of \eqref{eq:nlsR} with $\lambda=\overline{\lambda}$ and $p=p_n$ (and $\rho=1$), and $a_n\in\R$ is a suitable number depending on $n$. It is then easy to see that 
	\begin{equation}
	\label{eq:an}
	\liminf_{n\to+\infty}a_n>-\infty\,,
	\end{equation}
	because if this were not the case we would obtain again a contradiction as in \eqref{eq:assurdo}. By \eqref{eq:an}, $p_n\to6^+$ and the properties of $\varphi_{\overline{\lambda},p_n}$, it thus follows that the restriction of $u_n$ to $\HH$ converges strongly in $H^1(\HH)$ to $\varphi_{\overline{\lambda},6}(\cdot+a)$, for a suitable $a\in\R\cup\{+\infty\}$ (if $a=+\infty$, then $u_n$ tends to 0 strongly in $H^1(\HH)$). By uniqueness of the limit, this means that $u(\cdot)\equiv\varphi_{\overline{\lambda},6}(\cdot+a)$ on $\HH$ and, in particular,
	\begin{equation}
	\label{eq:convH}
	\|u_n\|_{L^2(\HH)}\to\|u\|_{L^2(\HH)},\qquad\|u_n\|_{L^{p_n}(\HH)}\to\|u\|_{L^6(\HH)}\qquad\text{as }n\to+\infty\,.
	\end{equation}
	Moreover, since $u_n\to u$ in $L^\infty_{loc}(\G)$, then $u_n(x)\to u(x)$ almost everywhere on the compact core $\mathcal{K}$ as $n\to+\infty$. Since $\K$ is a set of finite measure by assumption and $u_n$ is uniformly bounded from above (by the uniform boundedness in $H^1(\G)$), by dominated convergence  we have
	\[
	\|u_n\|_{L^{p_n}(\K)}\to\|u\|_{L^6(\K)}\qquad\text{as }n\to+\infty\,.
	\]
	Together with \eqref{eq:convH}, the strong convergence of $u_n$ to $u$ in $L^2(\K)$ and weak lower semicontinuity, this yields
	\[
	\begin{split}
	&\|u_n\|_2\to\|u\|_2,\qquad\|u_n\|_{p_n}\to\|u\|_6\qquad\text{as }n\to+\infty\\
	&\liminf_{n\to+\infty}\|u_n'\|_2\geq\|u'\|_2\,,
	\end{split}
	\]
	that is
	\[
	\sigma:=\frac{\|u'\|_2^2+\overline{\lambda}\|u\|_2^2}{\|u\|_6^6}\leq\liminf_{n\to+\infty}\frac{\|u_n'\|_2^2+\overline{\lambda}\|u_n\|_2^2}{\|u_n\|_{p_n}^{p_n}}=1\,,
	\]
	the last equality being a consequence of $u_n\in\NN_{\overline{\lambda},p_n}(\G)$ for every $n$. Noting that $\sigma^{\frac14}u\in\NN_{\overline{\lambda},6}(\G)$ by definition, we obtain
	\[
	\begin{split}
	\inf_{v\in\NN_{\overline{\lambda},6}(\G)}J_{\overline{\lambda},6}(v,\G)\leq J_{\overline{\lambda},6}\left(\sigma^{\frac14}u\right)=&\,\left(\frac12-\frac16\right)\sigma^{\frac32}\|u\|_6^6\leq \liminf_{n\to+\infty}\left(\frac12-\frac1{p_n}\right)\|u_n\|_{p_n}^{p_n}\\
	=&\,\liminf_{n\to+\infty}\inf_{v\in\NN_{\overline{\lambda},p_n}(\G)}J_{\overline{\lambda},p_n}(v,\G)=\inf_{v\in\NN_{\overline{\lambda},6}(\G)}J_{\overline{\lambda},6}(v,\G),
	\end{split}
	\]
	where we exploited also the continuity of the map $p\mapsto\inf_{v\in\NN_{\overline{\lambda},p}(\G)}J_{\overline{\lambda},p}(v,\G)$. The previous chain of inequalities shows that $\sigma=1$, $u\in\NN_{\overline{\lambda},6}(\G)$, $J_{\overline{\lambda},6}(u,\G)=\inf_{v\in\NN_{\overline{\lambda},6}(\G)}J_{\overline{\lambda},6}(v,\G)$ and that the convergence of $u_n$ to $u$ is strong in $H^1(\G)$. By Lemma \ref{lem:J} above, \cite[Theorem 1.3]{DST_MA} and the aforementioned results of \cite{AST_CMP}, this implies that $\|u\|_2^2=\mu_\R$ and $E(u,\G)=\inf_{v\in H_{\mu_\R}^1(\G)}E(v,\G)<0$. By strong convergence in $H^1(\G)$, this entails $E(u_n,\G)<0$ for sufficiently large $n$, providing the contradiction we seek and concluding the proof.
\end{proof}

We can now conclude our analysis by proving Theorem \ref{thm:prop_p}.

\begin{proof}[Proof of Theorem \ref{thm:prop_p}]
Let $\G$ be a noncompact graph with finitely many edges, at least one of which bounded, and such that every vertex is attached at an even number (possibly zero) of half-lines. Observe first that, on a graph with these properties, for every $\lambda>0$ there exists a positive solution $u_\lambda$ of \eqref{nlsG}. Indeed, since for every vertex $\vv_i\in\V_\G$ there exists $k_i\in\N$ such that $\vv_i$ is attached to $2k_i$ half-lines, we can group these half-lines in pairs and think of them as $k_i$ full lines crossing at the same point $\vv_i$.  Then, identifying $\vv_i$ with the origin of each of these $k_i$ lines and denoting by $\varphi_\lambda$ the unique solution of the nonlinear Schr\"odinger equation \eqref{eq:nlsR} on $\R$ (with $\rho=1$), it is readily seen that the function
\[
u_\lambda(x):=\begin{cases}
	\lambda^{\frac1{p-2}} & \text{for }x\in\K\\
	\varphi_\lambda(x+\tau_\lambda) & \text{for }x\text{ on each of the $k_i$ lines attached to $\vv_i$, for every $\vv_i\in\V_\G$},
\end{cases}
\]
with $\K$ the compact core of $\G$ and $\tau_\lambda>0$ such that $\displaystyle\varphi_\lambda(\tau_\lambda)=\lambda^{\frac1{p-2}}$, solves \eqref{nlsG} on $\G$. Note that $\tau_\lambda$ exists because
		\begin{equation*}
			\| \varphi_{\lambda} \|_{\infty}
			= \varphi_{\lambda}(0)
			= \left( \frac{\lambda p}{2} \right)^{\frac{1}{p-2}}
			> \lambda^{\frac{1}{p-2}}.
		\end{equation*}
To complete the proof of Theorem \ref{thm:prop_p} we are thus left to show that, varying $\lambda$, $u_\lambda$ has negative energy as soon as its mass is sufficiently large. Since $u_\lambda$ is made of a constant function on the compact core of $\G$ and a certain number, say $k$, of full copies of $\varphi_\lambda$, we have
		\begin{align*}
			\| u_{\lambda} \|_{2}^{2}
			&= |\K| \lambda^{\frac{2}{p-2}} + k \| \varphi_{\lambda} \|_{2}^{2}
			= |\K| \lambda^{\frac{2}{p-2}} + k \lambda^{\frac{6-p}{2(p-2)}} \| \varphi_1 \|_{2}^{2}\,,\\
			\| u_{\lambda} \|_{p}^{p}
			&= |\K| \lambda^{\frac{p}{p-2}} + k \| \varphi_{\lambda} \|_{p}^{p} 
			= |\K| \lambda^{\frac{p}{p-2}} + k \lambda^{\frac{p+2}{2(p-2)}} \| \varphi_1 \|_{p}^{p}\,,\\
			\| u_{\lambda}' \|_{2}^{2}
			&= k \| \varphi'_{\lambda} \|_{2}^{2}
			= k \lambda^{\frac{p+2}{2(p-2)}} \| \varphi'_1 \|_{2}^{2}\,,
		\end{align*}
		where we used the well-known relation $\displaystyle	\varphi_{\lambda}(x) = \lambda^{\frac{1}{p-2}} \varphi_1\left(\sqrt{\lambda}x\right).$
		Therefore,
		\begin{equation*}
			E(u_{\lambda},\G)
			= \frac{1}{2}\| u_{\lambda}' \|_{2}^{2} - \frac{1}{p} \| u_{\lambda} \|_{p}^{p}
			= k \left( \frac{\| \varphi'_1 \|_{2}^{2}}{2} - \frac{\| \varphi_1 \|_{p}^{p}}{p} \right) \lambda^{\frac{p+2}{2(p-2)}}
			- \frac{|\K|}{p} \lambda^{\frac{p}{p-2}}.
		\end{equation*}
		Since $p > 2$, one has $\frac2{p-2}>0$ and $\frac{p}{p-2} > \frac{p+2}{2(p-2)}$.
		Hence, if $\lambda$ is large enough (depending on $k$, $|\K|$ and $p$), then $\|u_\lambda\|_2$ is large and $E(u_{\lambda},\G) < 0$.
\end{proof}

\section*{Acknowledgements}
S.D. acknowledges that this study was carried out within the project E53D23005450006 ``Nonlinear dispersive equations in presence of singularities'' -- funded by European Union -- Next Generation EU  within the PRIN 2022 program (D.D. 104 - 02/02/2022 Ministero dell'Università e della Ricerca). This manuscript reflects only the author's views and opinions and the Ministry cannot be considered responsible for them. The work has been partially supported by the INdAM GNAMPA project 2023 “Modelli nonlineari
in presenza di interazioni puntuali”. This work has been carried out in the framework of the Project NQG (ANR-23-CE40-0005-01), funded by the French National Research Agency (ANR).
L. Jeanjean thank the ANR for its support.

\end{document}